\documentclass[12pt,a4paper,oneside]{amsart}
\usepackage{a4wide}
\usepackage{graphicx}
\usepackage{amsmath}
\usepackage{amsfonts}
\usepackage{amssymb}
\usepackage{stmaryrd}
\usepackage{wasysym}
\usepackage{dsfont}
\usepackage{graphicx}
\usepackage{enumitem}
\usepackage{hyperref}
\usepackage{cleveref}

\newcommand{\abs}[1]{\lvert#1\rvert}

\newcommand{\cali}[1]{\mathcal{#1}}
\newcommand{\bb}[1]{\mathbb{#1}}
\newcommand{\perms}{\mathfrak{S}}
\newcommand{\ibar}{{\bar{\imath}}}
\newcommand{\CRTexc}{\mathcal E_{\mathrm{CRT}}}
\DeclareMathOperator{\E}{\mathbb{E}}
\DeclareMathOperator{\idf}{\mathds{1}}
\DeclareMathOperator*{\Prob}{\mathbb{P}}

\DeclareMathOperator{\rank}{rank}

\DeclareMathOperator{\Leb}{Leb}
\DeclareMathOperator*{\argmin}{argmin}

\DeclareMathOperator{\perm}{perm}
\DeclareMathOperator{\Perm}{Perm}

\DeclareMathOperator{\subperm}{subperm}

\DeclareMathOperator{\supp}{supp}
\DeclareMathOperator{\Id}{Id}
\DeclareMathOperator{\diam}{diam}
\DeclareMathOperator{\width}{width}
\DeclareMathOperator{\height}{height}
\DeclareMathOperator{\Dirichlet}{Dirichlet}
\DeclareMathOperator{\Image}{Im}

\newtheorem{theorem}{Theorem}[section]
\newtheorem{lemma}[theorem]{Lemma}
\newtheorem{proposition}[theorem]{Proposition}
\newtheorem{corollary}[theorem]{Corollary}
\newtheorem{fact}[theorem]{Fact}
\theoremstyle{definition}
\newtheorem{definition}[theorem]{Definition}
\theoremstyle{remark}
\newtheorem{remark}[theorem]{Remark}
\newtheorem{observation}[theorem]{Observation}

\author{Micka\"el Maazoun}
\title{On the Brownian separable permuton}
\address{Universit\'e de Lyon -- \'Ecole Normale Sup\'erieure de Lyon -- Unit\'e de Math\'ematiques Pures et Appliqu\'ees}
\email{mickael.maazoun@ens-lyon.fr}
\date{April 1,2019}
\begin{document}

\begin{abstract}
	The Brownian separable permuton is a random probability measure on the unit square, which was introduced by Bassino, Bouvel, F\'eray, Gerin, Pierrot (2016) as the scaling limit of the diagram of the uniform separable permutation as size grows to infinity.
	We show that, almost surely, the permuton is the pushforward of the Lebesgue measure on the graph of a random measure-preserving function associated to a Brownian excursion whose strict local minima are decorated with i.i.d. signs. As a consequence, its support is almost surely totally disconnected, has Hausdorff dimension one, and enjoys self-similarity properties inherited from those of the Brownian excursion.  The density function of the averaged permuton is computed and a connection with the shuffling of the Brownian continuum random tree is explored.
\end{abstract}

\maketitle

\section{Introduction}
For $n\geq 1$, let $\perms_n$ be the set of permutations of $\llbracket 1,n\rrbracket$, and $\perms = \sqcup_{n\geq 1}\perms_n$.
We use the one line notation $\sigma = (\sigma(1)\,\sigma(2)\,\cdots\, \sigma(n))$ for $\sigma \in \perms_n$.
A \emph{pattern} in a permutation $\sigma\in \perms_n$ induced by the indices $1\leq i_1<\ldots i_k \leq n$ is the permutation $\pi \in \perms_k$ that is order-isomorphic to the word $(\sigma(i_1),\ldots,\sigma(i_k))$. 
The \emph{density} of the pattern $\pi\in \perms_k$ in $\sigma\in \perms_n$ is the proportion of increasing $k$-uples in $\llbracket 1,n\rrbracket$ that induce $\pi$ in $\sigma$.
A \emph{class} of permutations is a subset of $\perms$ that is stable by pattern extraction, and is characterized by the pattern avoidance of some minimal family of permutations called its \textit{basis} \cite[5.1.2]{Bona}.
There is a large literature on the asymptotics of the pattern densities and diagram shape of a large typical permutation in several classes. This type of results can, to some extent, be encoded as convergence to a \emph{permuton}. In \cite{main} (to which we refer the reader for an extensive review of literature), Bassino, Bouvel, F\'eray, Gerin and Pierrot studied the class of \emph{separable permutations} and showed the convergence of a uniform large separable permutation to a \emph{Brownian separable permuton}, of which the present paper is a detailed study. Let us start with a few definitions.
\subsection{Limits of permutations}
A probability measure on the unit square $[0,1]^2$ is called a \textit{permuton} if both its marginals on $[0,1]$ are uniform.
With every permutation $\sigma\in \perms_n$ we associate a permuton $\mu_\sigma$ by setting $\mu_\sigma(dxdy) = n\idf\left[\sigma(\lfloor xn \rfloor) = \lfloor yn \rfloor\right] dxdy$. The set of permutons is equipped with the weak convergence of probability measures, which makes it compact.
A sequence of permutations $(\sigma_n)_n$ is said to converge to a permuton $\mu$ if and only if $\mu_{\sigma_n}$ converges weakly to $\mu$. This theory was introduced by Hoppen, Kohayakawa, Moreira, R\'ath, Sampaio in \cite{hoppen}, where it is shown that convergence of a sequence of permutations to a permuton is equivalent to convergence of all pattern densities. As a result, permutons can be alternatively constructed as the completion of the space of permutations w.r.t. convergence of all pattern densities. 
This theory is similar to graphons as limits of dense graphs, and unifies the study of the limit shape of the permutation diagram with that of the limit of pattern densities.

\subsection{The case of separable permutations}
\label{BP_sec:intro_separable}
A permutation is separable if it does not have $(2413)$ and $(3142)$ as an induced pattern.
Separable permutations were introduced in \cite{Bose1993},
but appeared earlier in the literature \cite{AvisMonroe81,shapiro1991bootstrap}.
They are counted by the large Schr\"oder numbers: $1, 2, 6, 22, 90, 394, \ldots$
and enjoy many simple characterizations \cite{Bose1993,AvisMonroe81,shapiro1991bootstrap,Ghys}

The one most relevant to this paper is in terms of trees.
A signed tree $t$ is an rooted plane tree whose internal nodes are decorated with signs in $\{\oplus, \ominus\}$.
We label its leaves $1,\ldots,k$ according to the natural ordering of $t$.
The signs can be interpreted as coding a different ordering of the rooted tree $t$: we call $\tilde t$ the tree obtained from $t$ by reversing the order of the children of each node with a minus sign.
The order of the leaves is changed by this procedure, and we set $\sigma(i)$ to be the position in $\tilde t$ of the leaf $i$.
We call $\perm(t)$ this permutation $\sigma \in \perms_k$. 
It turns out \cite[Lemma 3.1]{Bose1993} that separable permutations are exactly the ones that can be obtained this way.
\begin{figure}[h]
	\centering
	\includegraphics{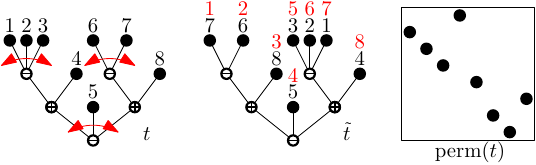}
	\caption{The permutation associated to a signed tree.}
\end{figure}

The article \cite{main} shows that separable permutations have a permuton limit in distribution, yielding the first example of a nondeterministic permuton limit of a permutation class. The representation by signed trees is fundamental in their proof.
\begin{theorem}[theorem 1.6 of \cite{main}]\label{BP_thm:bassinosimple}
	If $\sigma_n$ is a uniform separable permutation of size $n$, then $\mu_{\sigma_n}$ converges in distribution, in the weak topology, to a non-deterministic permuton $\mu^{1/2}$ called the Brownian separable permuton of parameter $1/2$.
\end{theorem}

This result comes with a characterization of $\mu^{1/2}$  (which we recall in section 2) which suggests that it can be realized as a measurable functional of a \textit{signed Brownian excursion} (see \cref{BP_rk:cesar}). The authors of \cite{main} left this, along with the study of the support of $\mu^{1/2}$, as open questions that the present paper aims at addressing.

Let us mention that \cref{BP_thm:bassinosimple} was generalized in \cite{main2,main3} by the same authors along with the present author to various families of permutation classes. These results yield, among others, a one-parameter family $(\mu^p)_{p\in (0,1)}$ of possible limits, called the \textit{biased Brownian separable permutons}.
We set our paper in this generality and fix once and for all $p\in (0,1)$. We postpone a precise definition of $\mu^p$ to \cref{BP_sec:defs}.

\subsection{The signed Brownian excursion}
We call \textit{continuous excursion} a nonnegative function $g:[0,1]\to \mathbb R_+$ that is positive on $(0,1)$. The \textit{inner local minima} of $g$ are the points of $(0,1)$ in which $g$ is locally minimal, and we say that $x\in (0,1)$ is \textit{not a one-sided minimum} of $g$ if
\[\forall \epsilon>0, \exists x_1 \in (x-\epsilon,x), x_2 \in (x,x+\epsilon) \text{ s.t. } g(x_1)<g(x) \text{ and } g(x_2)<g(x).\]
A \textit{CRT excursion} is a continuous function $g : [0,1] \to \bb R_+$ such that:
\begin{enumerate}[label =(CRT\arabic*), ref=(CRT\arabic*),leftmargin=*]	
	\item the inner local minima of $g$ are dense in $[0,1]$, \label{BP_crtdensitybp}
	\item %if $t_1<t_2$ are two strict local minima and $g(t_1)=g(t_2)$, then
	%$\inf_{t_1<t<t_2}g(t) < g(t_1)$.
	the values at the inner local minima are all different, \label{BP_crtbinary}
	\item the set of times that are not one-sided minima has Lebesgue measure 1. \label{BP_crtasleaf}
\end{enumerate}
In a CRT excursion, all inner local minima are necessarily strict local minima, and hence countable. It will be useful for our purposes to enumerate them in a well-defined manner.
\begin{definition}\label{BP_def:mesenum}
	A \textit{measurable enumeration} is a sequence $(b_i)_{i\in \mathbb N}$ of functions from the set $\CRTexc$ of CRT excursions to $[0,1]$ such that
	\begin{enumerate}[label =(ME\arabic*), ref=(ME\arabic*),leftmargin=*]
		\item for every $g\in \CRTexc$, $i\mapsto b_i(g)$ is a bijection between $\mathbb N$ and the inner local minima of $g$,\label{BP_mesenum1}
		\item for every $i\in \mathbb N$, $g\mapsto b_i(g)$ is measurable,\label{BP_mesenum2}
		\item the function which maps $(g,u,v)\in \CRTexc\times[0,1]^2$ to $i\in \mathbb N$ if $b_i \in (u,v)$ is the unique point in $[u,v]$ in which the minimum of $g$ on $[u,v]$ is reached, and $\infty$ otherwise, is measurable.\label{BP_mesenum3}
	\end{enumerate}
\end{definition}
We fix once and for all a measurable enumeration (see \cref{BP_sec:defs} for an explicit construction of one, which comes from \cite{main}).
We call \textit{signed excursion} a pair $(g,s)$, where $g$ is a CRT excursion and $s$ is a sequence in $\{\oplus,\ominus\}^\mathbb N$. The sign $s_i$ is to be considered as attached to the inner local minimum $b_i$.

Let $(g,s)$ be a signed excursion. If $x< y\in [0,1]$, we say that $x$ and $y$ are $g$-comparable if and only if the minimum of $g$ on $[x,y]$ is reached at a unique point which is a strict local minimum $b_i \in (x,y)$.
In this case, if $s_i=\oplus$, we say $x\lhd_g^s y$, otherwise $y \lhd_g^s x$. 

The relation $\lhd_g^s$ is a strict order, but it is not total. However, two distinct points which are not one-sided minima are always $g$-comparable, hence $\lhd_g^s$ is total on a set of measure $1$. See \cref{BP_lem:comparability} for the proof of these claims. Moreover we will see later (\cref{BP_sec:intro_f}) a natural extension to a total preorder on $[0,1]$.

In what follows, we consider the signed excursion $(e,S)$, where $e$ is a the normalized Brownian excursion, and $S$ is an independent sequence of independent signs with bias $p$, that is probability $p$ of being $\oplus$ and $1-p$ of being $\ominus$.
It is the main ingredient in building $\mu^p$.
\subsection{Construction of the permuton}
If $(g,s)$ is a signed excursion, we define 
\begin{equation}\varphi_{g,s}(t) = \Leb\{u\in [0,1], u\lhd_g^s t\},\quad t\in[0,1]
\label{BP_eq:defvarphi}
\end{equation}
and 
\[\mu_{g,s} = (\Id,\varphi_{g,s})_*\Leb.\] Here $H_*\nu$ denotes the pushforward measure $\nu(H^{-1}(\cdot))$, whenever $H$ and $\nu$ are respectively a measurable function and a measure defined on the same space. The reader may report to \cref{BP_fig:eandf}, disregarding for now the vertical excursion $\tilde e$, to see a simulation of $e,S$ and $\varphi_{e,S}$.
Our main theorem is the following:
\begin{theorem}\label{BP_thm:main}
	The maps $(t,g,s) \mapsto \varphi_{g,s}(t)$ and $(g,s)\mapsto \mu_{g,s}$ are measurable, and
	%if $e$ is a normalized brownian excursion and $S$ a iid sequence of balanced signs,
	the random measure $\mu_{e,S}$ is distributed like $\mu^p$, the biased Brownian separable permuton of parameter $p$.
\end{theorem}

This theorem is proved in \cref{BP_sec:varphi}, along with a corollary which shows that the convergence of \cref{BP_thm:bassinosimple} can be rewritten without permutons, only in terms of functional convergence.
To any permutation $\sigma \in \perms_n$, we associate a c\`adl\`ag, piecewise affine, measure-preserving function $\varphi_\sigma :  [0,1] \to [0,1]$ with 
$\varphi_\sigma(x) = \frac 1 n (\sigma (\lfloor nt \rfloor +1)-1) + \frac 1 n \{ nt \}$. 
\begin{corollary}\label{BP_cor:cvoffunctions}
	Let $\sigma_n$ be a random permutation in  $\perms_n$ for every $n\in \mathbb N$. If $\mu_{\sigma_n}$ converges in distribution to $\mu^p$, then for every $q\in [1,\infty)$, we have the convergence in distribution in the space $L^q([0,1])$:
	\[\varphi_{\sigma_n} \xrightarrow[n\to\infty]{d} \varphi_{e,S} \]
\end{corollary} 

\subsection{Properties of the permuton}
This continuum construction allows us to derive several properties of $\mu^p$.
In \cref{BP_sec:onedim}, we prove the following result.
\begin{theorem} \label{BP_thm:dimension}
	Almost surely, the support of $\mu^p$ is totally disconnected, and its Hausdorff dimension is 1 (with one-dimensional Hausdorff measure bounded above by $\sqrt 2$). 
\end{theorem}

The claim that the Hausdorff dimension is $1$ also comes as a special case of a result of Riera \cite{Riera}: any permuton limit in distribution of random permutation in a proper class, if it exists, almost surely has a support of Hausdorff dimension $1$. 

In \cref{BP_sec:ss}, we show that $\mu^p$ inherits the self-similarity properties of $e$, in that $\mu^p$ contains a lot of rescaled distributional copies of itself. 
In particular, we get the following theorem, illustrated in \cref{BP_fig:ss}, which states that $\mu^p$ can be obtained by cut-and-pasting three independent Brownian separable permutons.

\begin{theorem}\label{BP_thm:ss}
	Let $(\Delta_0,\Delta_1,\Delta_2)$ be a random variable of $ \Dirichlet(\tfrac 1 2,\tfrac 1 2,\tfrac 1 2)$ distribution. 
	Let $\mu_0, \mu_1, \mu_2$ be independent and distributed like $\mu^p$, and
	% conditionally on these three random measures, $(X_0,Y_0)$ has distribution $ \mu_0$, $(X_1, Y_1)$ has distribution $\mu_1$ and $(X_2,Y_2)$ has distribution $\mu_2$.
	conditionally on $\mu_0$, let $(X_0,Y_0)$ be a random point of distribution $\mu_0$.
	Let $\beta$ be an independent Bernoulli r.v. of parameter $p$.
	We define the piecewise affine maps of the unit square into itself:
	\begin{equation}\begin{aligned}\label{BP_eq:definitiontheta}
	& \theta_0(x,y) & = (\eta_0(x), \zeta_0(y))& = \Delta_0(x,y) + (1-\Delta_0)(\idf_{[x>X_0]},\idf_{[y>Y_0]})  \\
	& \theta_1(x,y) & = (\eta_1(x), \zeta_1(y))&= \Delta_1(x,y) +\Delta_0(X_0,Y_0) +\Delta_2(0,\beta)              \\
	& \theta_2(x,y) & = (\eta_2(x), \zeta_2(y)) &= \Delta_2(x,y) +\Delta_0(X_0,Y_0) +\Delta_1(1,1-\beta)           
	\end{aligned}\end{equation}
	%Define $\mu = \Delta_0 \theta_0*\mu_0 + \Delta_1 \theta_1*\mu_1 + \Delta_2 \theta_2*\mu_2$,  $(X_l,Y_l)=\theta_1(X_1,Y_1)$ and $(X_r,Y_r)=\theta_2(X_2,Y_2)$.
	%Then $\mu$ is distributed like $\mu^p$ and, conditional on $\mu$, $(X_l,Y_l)$ and $(X_r,Y_r)$ are independently distributed according to $\mu$ conditioned on $\{X_l<X_r\}$
	Then 
	\begin{equation}\Delta_0 \theta_0{}_*\mu_0 + \Delta_1 \theta_1{}_*\mu_1 + \Delta_2 \theta_2{}_*\mu_2 \stackrel d= \mu^p,
	\label{BP_eq:ss}
	\end{equation}
\end{theorem}

\begin{figure}[htb]
	\centering
	
	\includegraphics{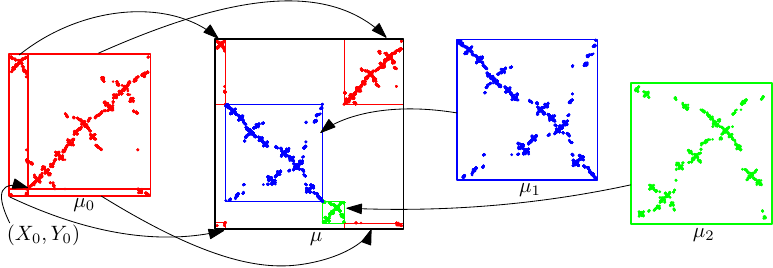}
	\caption{The construction of $\mu$ from three independent permutons distributed like $\mu$. Here $\beta = 0$ and $(\Delta_0,\Delta_1,\Delta_2) \approx (0.4,0.5,0.1)$.\label{BP_fig:ss}}
\end{figure}

We believe that a result by Albenque and Goldschmidt \cite{albenque2015} about the Brownian CRT can be adapted to show that the \textit{distributional identity} \eqref{BP_eq:ss} characterizes $\mu^p$ (see \cref{BP_rk:ag}.)

Finally, our construction allows us to compute the averaged permuton $\E\mu^p$, obtained by taking $\E\mu^p(A) = \E[\mu^p(A)]$ for every Borel set $A$. We get the following result.
\begin{theorem}\label{BP_thm:expectation}
	The permuton $\E\mu^p$ is the measure $\alpha(x,y)dxdy$, where $\alpha(x,y)$ equals
	\[
	\int_{\max(0,x+y -1)}^{\min(x,y)} \frac{3p^2(1-p)^2 da }
	{2\pi(a(x-a)(1-x-y+a)(y-a))^{3/2}{\left(\frac{p^2}{a}+\frac{(1-p)^2}{(x-a)}+\frac{p^2}{(1-x-y+a)}+\frac{(1-p)^2}{(y-a)}\right)^{5/2}}}.
	\]
\end{theorem}

Plots for different values of $p$ are provided on \cref{BP_fig:plot_expectation}. The function $\alpha_p$ is a priori a rather complicated elliptic integral involving the root of a polynomial of degree 3 in $a$. However the case $p=1/2$ is special: first of all $\alpha_{1/2}$ has all the symmetries of the square, so that we may restrict to $0\leq x\leq \min(y ,1-y)$. Furthermore thanks to some cancellations, the polynomial under the root is only of degree $2$,
%cancellations lead to the following rewriting:
%\[
%\alpha_{1/2}(x,y) = \frac 3 \pi \int_{0}^{x} \frac{a(x-a)(1-x-y+a)(y-a)}
%{\left(xy(1 - x)(1 - y) - (a - xy)^2\right)^{5/2}}da.
%\]
and the integral can be solved for instance with a computer algebra system, yielding
\begin{align}
\alpha_{1/2}(x,y) &= \frac 1 {\pi} (\beta(x,y) + \beta(x,1-y)),\quad 0\leq x\leq \min(y ,1-y),\label{BP_eq:alpha12}\\
\text{where }\beta(x,y) &= \frac{3 x y-2 x-2 y+1}{(1-x) (1-y)}\sqrt{\frac{1-x-y} {x y}}+3\arctan\sqrt{\frac{x y}{1-x-y}}.\nonumber
\end{align}
The function $\alpha$ already appeared in a different form in the work of Dokos and Pak \cite{DokosPak} as the expected shape of doubly-alternating Baxter permutations. We give more details about this at the end of the introduction.
\begin{figure}[htb]
	\centering
	
	\includegraphics[width=0.28\linewidth]{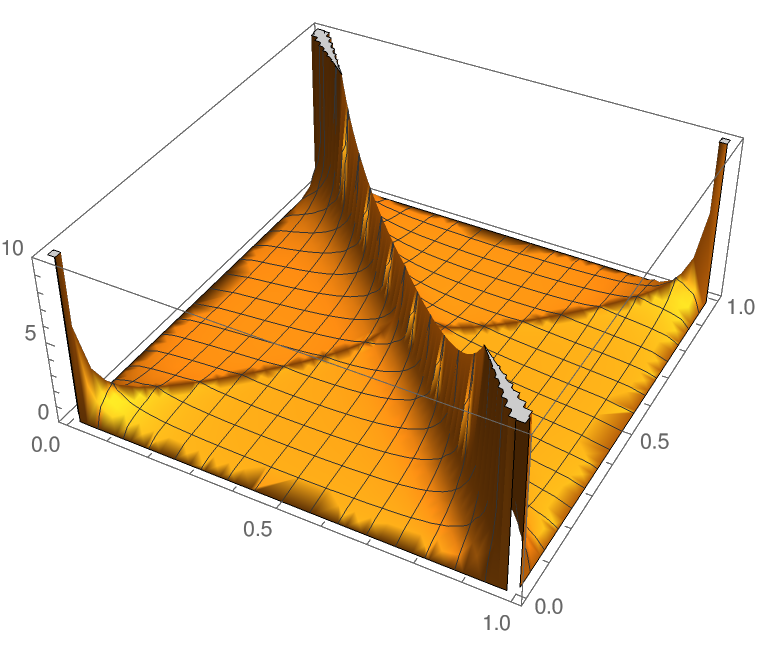}\includegraphics[width=0.28\linewidth]{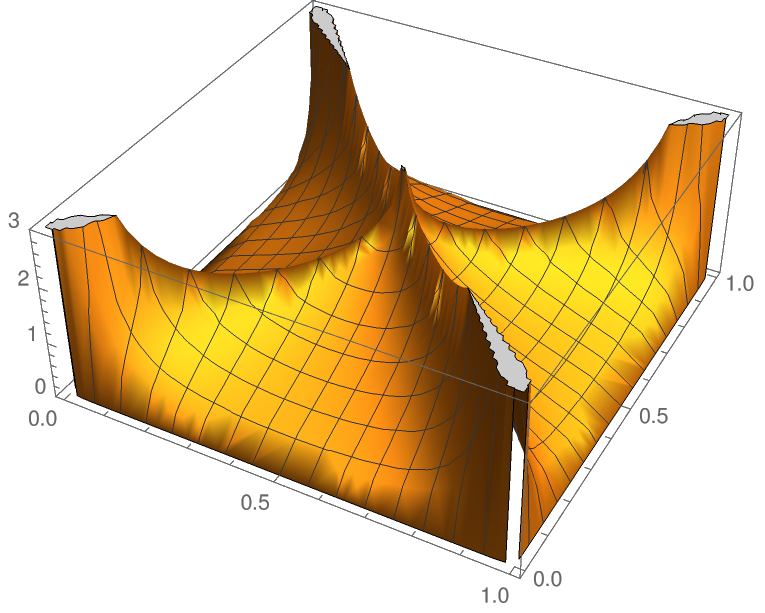}\includegraphics[width=0.28\linewidth]{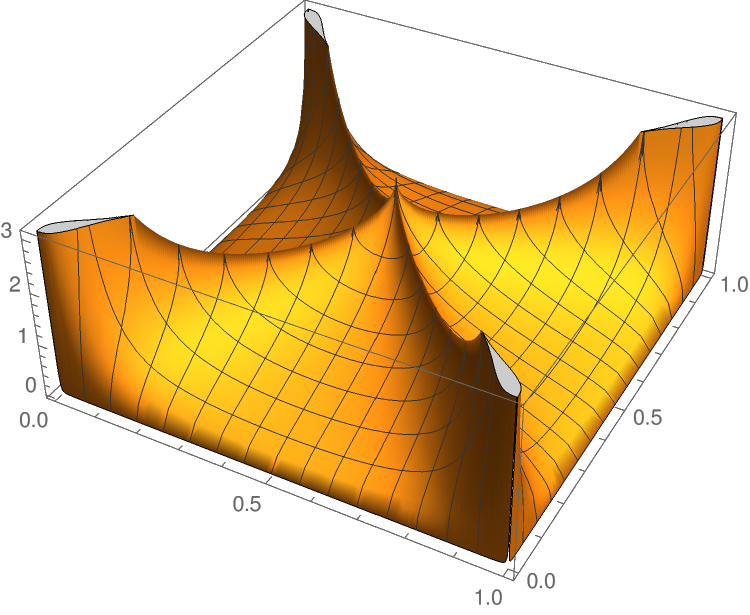}
	\caption{The function $\alpha_p$ for $p\in\{0.3,0.45,0.5\}$.\label{BP_fig:plot_expectation}}
\end{figure}

\subsection{Shuffling of continuous trees}
\label{BP_sec:intro_f}
%We introduced separable permutations as permutations that can be obtained by the shuffling of a discrete tree according to signs born by the internal nodes of the tree. It is known since Aldous \cite{aldous1993} that the Brownian excursion $e$ is the height process of a depth-first exploration of a random continuous tree $\mathcal T_e$ called the Brownian CRT. Since inner minima of the $e$ correspond with branching points of $\mathcal T_e$, we can interpret $S$ as signs born by the branching points of $\mathcal T_e$. It is then natural to think of $\phi_{e,S}$ as the continuous "permutation" that gives the reordering of the points of $\mathcal T_{e}$ according to a 

Through a classical construction (which goes back to Aldous \cite{aldous1993}), a Brownian excursion $e$ encodes a continuous (rooted and ordered) tree $\mathcal T_e$ called the Brownian CRT. This encoding puts inner minima of $e$ in correspondence with branching points of $\mathcal T_e$, so that the pair $(\mathcal T_e,S)$ may be seen as a continuous signed tree. 

The next few results make this rigorous and explain how the random function $\varphi_{e,S}$ relates to the tree $(\mathcal T_e,S)$, much like separable permutations relate to signed trees. Those results, and the notation introduced here, are not needed for the rest of the paper, albeit the fact that $e$ encodes a tree is an idea that underlies most of the arguments of the paper.

We recall the construction of continuous trees from continuous excursions, in the formalism of Le Gall and Duquesne \cite{legall2005, duquesne}. Let $g$ be a continuous excursion. Set $d_g(x,y) = g(x)+g(y)-2\min_{[x,y]}g$ for $x,y\in[0,1]$. The function $d_g$ is a pseudo-distance. Identifying points $x,y\in[0,1]$ such that $d_g(x,y)=0$ yields a quotient metric space $(\cali T_g,d_g)$ with a continuous canonical surjection $p_g : [0,1]\mapsto \cali T_g$.
Let $\rho_g = p_g(0)$ be the root of $\mathcal T_g$, and define a total order $\leq_g$ on $\mathcal T_g$ by setting $x\leq_g y \iff \inf p_g^{-1}(x) \leq \inf p_g^{-1}(y)$. Define a probability measure $\lambda_g = p_g{}_*\Leb_{[0,1]}$. 
When $g=e$, we get the well-known Brownian CRT. 

\Cref{BP_sec:f} is devoted to the proof of the following theorem, illustrated in \cref{BP_fig:eandf}.
\begin{theorem}
	\label{BP_thm:f}
	There exists a random CRT excursion $\tilde e$, defined on the same probability space as $(e,S)$, with the following properties: 
	\begin{enumerate}
		\item The excursion $\tilde e$ has the distribution of a normalized Brownian excursion, with the same field of local times at time 1 as $e$.
		\item Almost surely, the function $\varphi_{e,S}$ is an isometry between the pseudo-distances $d_e$ and $d_{\tilde e}$. 
		In particular, $\tilde e \circ \varphi_{e,S} = e$.
	\end{enumerate}
\end{theorem}

\begin{figure}[htb]
	\centering
	\includegraphics{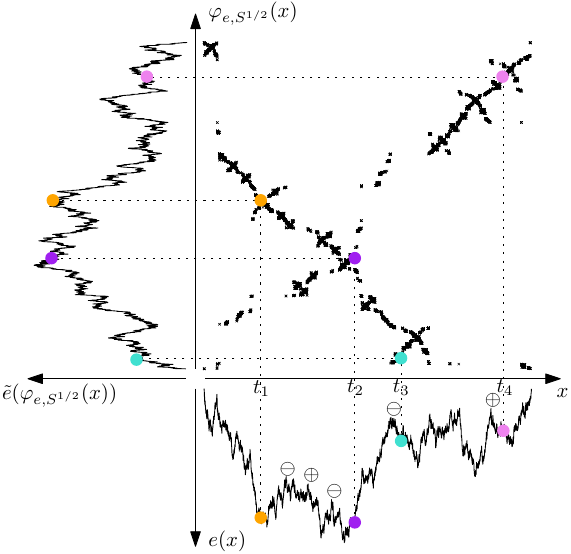}
	\caption{A realization of $(e, S)$ (here $p=1/2$), and the associated functions $\varphi_{e, S}$ and $\tilde e$, highlighting the property $\tilde e \circ \varphi_{e, S} = e$. Four points $t_1<\ldots<t_4$ are specified.\label{BP_fig:eandf}}
\end{figure}

This result has an interpretation in terms of shuffling of continuous trees, mirroring the construction of separable permutations described in \cref{BP_sec:intro_separable}.

When $g$ is a CRT excursion,
the construction of $\cali T_g$ puts the inner local minima of $g$ in bijection with the branching points of $\mathcal T_g$.
Hence, when $(g,s)$ is a signed excursion,
the order $\leq_g^s$ can be defined on the tree $\cali T_g$
by inverting at all branching points with a minus sign, as follows.
Let $x,y\in \cali T_g$ such that $x\leq_g y$. 
If there exists a strict local minimum $b_i$ such that
$\sup p_g^{-1}(x)< b_i < \inf p_g^{-1}(y)$, with $g(b_i) = \inf\{g(t), \sup p_g^{-1}(x)\leq t \leq \inf p_g^{-1}(y)\}$, and $s(b_i) = \ominus$, then set $x \geq_g^s y$. 
Otherwise, set $x\leq_g^s y$.
This defines a total order compatible with the relation on $[0,1]$ defined in the previous section: whenever $x$ and $y$ are $g$-comparable, then $x\lhd_g^s y \iff p_g(x) <_g^s p_g(y)$. This construction is illustrated in \cref{BP_fig:deuxarbres}.

\begin{figure}[t]
	\centering
	\includegraphics{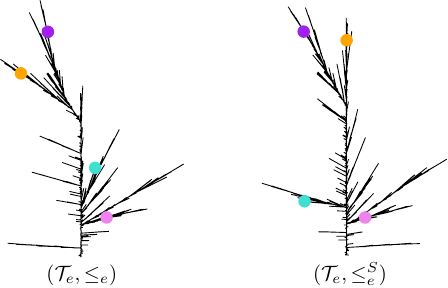}
	\caption{The tree $\mathcal T_e$, drawn according to the two orders $\leq_e$ and $\leq_e^S$. Four point have been marked. The data is the same as in \cref{BP_fig:eandf}.\label{BP_fig:deuxarbres}}
\end{figure}

This allows us to give an interpretation of \cref{BP_thm:f} in terms of trees. If we consider the tree $(T_{\tilde e},d_{\tilde{e}},\rho_{\tilde{e}},\leq_{\tilde{e}},\lambda_{\tilde e})$, \cref{BP_thm:f}(2) says that, for $x,y\in [0,1]$, $d_e(x,y) = 0 \iff d_{\tilde e} ( \varphi_{e, S}(x), \varphi_{e, S}(y)) = 0$. We deduce that $p_e(x) = p_e(y) \iff p_{\tilde e}(\varphi_{e, S}(x))= p_{\tilde e}(\varphi_{e, S}(y))$, which implies that the map $\varphi_{e,S}$ factorizes through $p_e$ and $p_{\tilde e}$, that is
there is a unique map $\jmath:\mathcal T_e \to \mathcal T_{\tilde e}$ such that $ \jmath \circ p_e  =  p_{\tilde e}\circ\varphi_{e, S}$. 
It is immediate than $\jmath$ is an isometry $(\mathcal T_e,d_e) \leftrightarrow (\mathcal T_{\tilde e},d_{\tilde{e}})$. 
Moreover, $\jmath$ maps the root of $\mathcal T_e$ to the root of $\mathcal T_{\tilde e}$, is measure preserving and increasing w.r.t. $(\leq_{e}^{S}, \leq_{\tilde e})$. This discussion can be summarized in the following corollary of \cref{BP_thm:f}.
\begin{proposition}
	The map $\jmath: \mathcal T_e \leftrightarrow \mathcal T_{\tilde e}$ provides an isomorphism (of pointed, ordered, measured metric spaces) between the tree $(T_e,d_e,\rho_e,\leq_{e}^{S},\lambda_e)$ and the Brownian CRT $(T_{\tilde e},d_{\tilde{e}},\rho_{\tilde{e}},\leq_{\tilde{e}},\lambda_{\tilde e})$ constructed from the Brownian excursion $\tilde e$.
\end{proposition}

Combining this with the result of Duquesne on the uniqueness of coding functions of trees \cite[Thm 1.1]{duquesne}, we directly get an abstract construction of $\mu_{e, S}$.
\begin{proposition}
	Almost surely, the functions $\tilde e$ and $\jmath$ are uniquely determined by the fact that $\tilde e$ is continuous and $\jmath$ is an isomorphism between $(T_e,d_e,\rho_e,\leq_{e}^S,\lambda_e)$ and $(T_{\tilde e},d_{\tilde{e}},\rho_{\tilde{e}},\leq_{\tilde{e}},\lambda_{\tilde e})$. Any function $\phi$ which verifies
	$p_{\tilde{e}} \circ \phi = \jmath \circ p_e$
	must coincide with $\varphi_{e,S}$ on a set of measure 1, hence still verifies $\mu_{e,S} = (\Id,\phi)_*\Leb$.
\end{proposition}

\subsection{Comments and perspectives}
%Let us mention that permutons are known in the statistics literature as \textit{two-dimensional copulas}, and in the ergodic theory/optimal transportation literature under the name \textit{doubly-stochastic measures}.
%Doubly-stochastic measures whose support is in the graph of a function have been investigated a
%An interesting corollary of \cref{BP_thm:main} is the fact that almost surely, $\mu^p$ is an extremal doubly-stochastic measure.

Let us mention another natural family of permutations: the doubly-alternating Baxter permutations, which are also the doubly-alternating separable permutations \cite{Ouchterlony}, and are counted by the Catalan numbers. The fact that they enjoy a tree decomposition similar to separable permutations, along with simulations \cite{DokosPak}, allows to boldly conjecture that they converge to the Brownian separable permuton of parameter $1/2$.

Dokos and Pak \cite[Thm 1.1]{DokosPak} compute the expected shape of doubly-alternating Baxter permutations: their result implies that for every Borel subset $A$ of the unit square, if $\sigma_n$ is a uniform doubly-alternating permutation of size $n$, then
$\E[\mu_{\sigma_n}(A)] \to \int_A \psi$, where $\psi$ has symmetries of the square and $\psi(x,y) =  \frac{1}{4\pi} \int_0^xdu\int_0^{x-u} \frac{dv}{[(u+v)(y-v)(1-y-u)]^{3/2}}$ for $0\leq x\leq y\wedge 1-y$. We can show that this function is the same as the one we computed for the expectation of the Brownian permuton of parameter $1/2$, further strengthening the conjecture. Indeed,
\begin{align*}\psi(x,y) 	&= \frac{1}{4\pi} \int_0^xdu\int_0^{u} \frac{dv}{[u(y-v)(1-y-u+v)]^{3/2}}\\
&= \frac{1}{4\pi} \int_0^xdu\left[\frac{2(-u+2 v-2 y+1)}{(u-1)^2 u^{3/2}\sqrt{(y-v) (1-y-u+v)}}\right]_{v=0}^{v=u}\\
&= \frac{1}{\pi} \int_0^x(\gamma(u,y) + \gamma(u,1-y))du
\end{align*}
where $\gamma(x,y) = \frac {x+2y-1}{2(1-x)^2x^{3/2}\sqrt{(y-x)(1-y-x)}}$.
We recall the definition of $\alpha_{1/2}$ and $\beta$ from \eqref{BP_eq:alpha12}. We can check that $\partial_{x}\beta(x,y) = \gamma(x,y)$, implying that $\psi = \alpha_{1/2}$. 

As already mentioned, the article \cite{main2} considers \textit{sub\-stitution-closed classes}, which are natural generalizations of the class of separable permutations. Depending on the class, several possible limits appear, among which are the $\mu^p$ for $p$ possibly different from $1/2$.
Another family of possible limits is \textit{the $\alpha$-stable permuton driven by $\nu$}, for $\alpha\in(1,2)$ and $\nu$ itself a random permuton.
We believe a continuum construction similar to the one presented here is possible, by considering a $\alpha$-stable tree, with an independent copy of $\nu$ at each branching point, driving the reordering of the (countably infinite number of) branches stemming from that point. 
%We do believe that the support would still be almost surely of Hausdorff dimension 1 in that case.

The structure of the paper is as follows. \Cref{BP_sec:defs} contains various definitions that will be needed in the rest of the paper, notably the definition of $\mu^p$ and a characterization through its finite-dimensional marginals that highlights the link with the signed excursion. \Cref{BP_sec:varphi} contains the proof of \cref{BP_thm:main}, along with some facts about the random function $\varphi_{e,S}$ that are reused later. \Cref{BP_sec:onedim,BP_sec:ss,BP_sec:expectation} are respectively devoted to the proofs of \cref{BP_thm:dimension,BP_thm:ss,BP_thm:expectation}, and \Cref{BP_sec:f} to the one of \cref{BP_thm:f}.

\section{Definitions}
\label{BP_sec:defs}
First we set a few notations : if $x_1,\ldots,x_k$ are strictly comparable elements of an ordered set $(E,\leq)$, then $\rank_\leq(x_1,\ldots,x_k)$ is the permutation $\alpha$ such that $\alpha(i)<\alpha(j) \iff x_i<x_j$ for every $1\leq i,j \leq k$. The sequence $(x_{\alpha^{-1}(1)}<\ldots<x_{\alpha^{-1}(k)})$ is called the order statistic of $(x_1,\ldots,x_k)$ and denoted $(x_{(1)}<\ldots<x_{(k)})$.
\subsection{Marginals of a permuton}
In this section we want to give a tractable definition of the random permuton $\mu^p$. This will take the form of a characterization through its finite-dimensional marginals, which we define now.
If $k\geq 1$ and $\mu$ is a random permuton, let $\subperm_k(\mu) = \rank(Y_1,\ldots Y_k) \circ \rank(X_1,\ldots X_k)^{-1} \in \perms_k$, where conditionally on $\mu$, the $(X_i,Y_i)$ for $i\in \llbracket 1,k\rrbracket$ are independent and distributed according to $\mu$. 
Then the distribution of $\subperm_k(\mu)$ is called the \textit{$k$-dimensional marginal} of $\mu$. 
The interest of this definition lies in the following result, which is an extension of the main theorem of \cite{hoppen} to sequences of random permutations.
\begin{proposition}[theorem 2.2 of \cite{main2}]
	\label{BP_prop:randompermuton}
	Let $\sigma_n$ is a sequence of (possibly random) permutation whose size goes to infinity. The following are equivalent
	\begin{enumerate}
		\item As $n\to\infty$, $\mu_{\sigma_n}$ converges in distribution to some random permuton $\mu$.
		\item For every $k\geq 1$, the uniform pattern of length $k$ in $\sigma_n$, denoted $\sigma_n^{(k)}$, converges in distribution, as $n\to\infty$ to some random permutation $\rho_k\in \perms_k$.
	\end{enumerate} 	
	In this case, the law of $\mu$ is characterized by $\subperm_k(\mu) \stackrel d= \rho_k$ for $k\geq 1$.
\end{proposition}

This is indeed the result used by \cite{main}, \cite{main2} and \cite{main3} to prove permuton convergence. 
As a result, the distribution of $\subperm_k(\mu^p)$ for every $k$ is obtained as follows (see \cite[prop. 9.1]{main} and \cite[def. 5.1]{main2})
\begin{definition}\label{BP_prob:bassino1}
	The permuton $\mu^p$ is determined by the relations
	\begin{equation}\label{BP_eq:caractmup1}
	\forall k\geq 1,\quad \subperm_k(\mu^p) \stackrel d= \perm(t_{k,p}),
	\end{equation}
	where $t_{k,p}$ is a uniform binary tree with $k$ leaves whose internal vertices are decorated with i.i.d. signs of bias $p$.
\end{definition}
In the rest of the section, we make apparent a connection with the signed Brownian excursion.

\subsection{A few facts about excursions}
We start by constructing a measurable enumeration as defined in \cref{BP_def:mesenum}. Let $(p_i,q_i)_{i\in \bb N}$ be a fixed enumeration of $\mathbb Q^2 \cap [0,1]$. Let $g$ be a CRT excursion.
For $i\geq 1$, define $w_i = \min \{t\in [p_i,q_i] : g(t) = \min_{[p_i,q_i]} g\}$, $i_0 = 0$, and for $k\geq 1$, set recursively
\[i_k = \inf \{i> i_{k-1}, w_i\in (p_i,q_i)\text{ , }w_i \notin \{ w_1, \ldots, w_{i_{k-1}}\}\}.\]
Finally, for $k\in \bb N$, set $b_k(g) = w_{i_k}$. 

\begin{lemma} This construction defines a measurable enumeration.\end{lemma}
\begin{proof}
	It is immediate that all inner local minima will appear in the sequence $(w_i)_i$. The way the subsequence $(b_i)_i$ of $(w_i)_i$ is chosen guarantees that only inner local minima appear, and only once, in $(b_i)_i$.
	
	Measurability of $g\mapsto b_i(g)$ for every $i$ follows from that of $g\mapsto w_i(g)$ and $k\mapsto i_k$.
	
	To prove \ref{BP_mesenum3} we see that thanks to \cref{BP_crtbinary}, the function $\CRTexc \times [0,1]^2\to \bb N \cup \{\infty\}$
	\[(g,x,y) \mapsto \min \left\{ i \in \bb N, g(b_i(g)) = \min_{[x,y]}g \text{ and } b_i(g)\in (x,y) \text{ and } \min_{[x,y]}g < \min(g(x),g(y)) \right\}\] 
	is a measurable functional that maps $(g,x,y)$ to $i\in \bb N$ whenever $b_i$ is the point in $(x,y)$ that is the only global minimum of $g$ on $[x,y]$, and $\infty$ if no such $i\in \bb N$ exists.
\end{proof}

We now collect a few facts about CRT excursions. In \cref{BP_sec:intro_f} we saw a that such functions encode continuous trees. So we borrow the vocabulary of trees in a way that is coherent with this encoding: the $x\in[0,1]$ which are not one-sided local minima are called \textit{leaves of $g$}. The $b_i$ for $i\in \bb N$ are called \textit{branching points of $g$} and are identified with $\bb N$. Set
\begin{align*}
a_i&=\sup\{t<b_i : g(t)=g(b_i)\},\\
c_i&=\inf\{t> b_i : g(t)=g(b_i)\};\\
h_i&=g(b_i)=g(c_i)=g(a_i).
\end{align*}
By definition, for $x\in (a_i,b_i) \cup (b_i,c_i)$, $g(x)\geq h_i$, defining two subexcursions at respectively the left and the right of $b_i$. 
We collect an immediate consequence of \ref{BP_crtbinary}, which states that these subexcursions are nested, with a binary tree structure (which comes from that of $\mathcal T_g$).
%Actually, we have $g(x)>h_i$ because of crtbinary and crtdensitybp.
%Remark that \ref{BP_crtdensitybp} implies \[g(t)= \sup_{i} h_i \idf_{[a_i,c_i]}(t).\]
\begin{lemma}\label{BP_lem:genealogyofbp}
	For every $i,j$ either $[a_i,c_i] \subset [a_j,c_j]$ or $[a_j,c_j] \subset [a_i,c_i]$ or $[a_i,c_i]\cap [a_j,c_j]=\emptyset$. 
	
	Furthermore,  if $[a_j,c_j] \subset [a_i,c_i]$ , then either $j=i$, $[a_j,c_j] \subset (a_i,b_i)$  or $[a_j,c_j] \subset (b_i,c_i)$.
\end{lemma}

If $x<y$ are $g$-comparable, the $b_i$ in which $g$ reaches its minimum between $x$ and $y$ is called the\textit{ most recent common ancestor} of $x$ and $y$. We extend this notion to branching points: if $[a_i,c_i] \cap [a_j,c_j] = \emptyset$, then $b_i$ and $b_j$ are $g$-comparable. We can always assume by symmetry that $b_i<b_j$ and call \textit{most recent common ancestor} of $i$ and $j$ the $k\in \bb N$ such that $[a_i,c_i] \subset (a_k,b_k)$ and $[a_j,c_j] \subset (b_k,c_k)$.

\subsection{Extraction of permutations and trees from a signed excursion.}
Let $(g,s)$ be a signed excursion. Recall that $x$ and $y$ are $g$-comparable if the minimum of $g$ on $[x,y]$ is reached at a unique point, and that point $b$ is a strict local minimum with $b\in(x,y)$. We start by collecting elementary facts on comparability.
\begin{lemma} Let $g$ be a CRT excursion and $s$ a sequence of signs.
	\begin{enumerate}
		\item Two leaves of $g$ are always $g$-comparable. Hence almost every pair of points in $[0,1]$ is $g$-comparable.
		\item The relation $\lhd_g^s$ is a strict partial order.
		\item For almost every $x,y\in[0,1]$, 
		\[x\lhd_g^s y \implies \varphi_{g,s}(x)\leq \varphi_{g,s}(y).\]
	\end{enumerate}
	\label{BP_lem:comparability}
\end{lemma} 
\begin{proof}
	The first claim is immediate: between two leaves, the minimum of the function $g$ cannot be reached at the endpoints and consequently is reached at some unique point, which is a inner minimum of $g$. 
	
	It is clear by definition that $\lhd_g^s$ is antisymmetric.
	To show transitivity suppose $x \lhd_g^s y \lhd_g^s z$. Let $b_i$ (resp. $b_j$)be the most recent common ancestor of $x$ and $y$ (resp. $y$ and $z$).
	Since $[a_i,c_i] \cap [a_j,c_j]$ contains $y$, it is nonnempty and by \cref{BP_lem:genealogyofbp}, either $[a_i,c_i] \subset [a_j,c_j]$ or the symmetric case. Let us treat only the first one.
	\begin{itemize}
		\item In the case $i=j$, then $x$ and $z$ must be on the same side of $b_i$, opposite $y$. Since $x\lhd_g^s y$, then $z\lhd_g^s y$, which is impossible.
		\item In the case $[a_i,c_i] \subset (a_j,b_j)$, then $x,y\in(a_j,b_j)$ and $z\in (b_j,c_j)$.
		\item  In the case $[a_i,c_i] \subset (b_j,c_j)$, then $x,y\in(b_j,c_j)$ and $z\in (a_j,b_j)$
	\end{itemize}
	In these last two cases, $x$ and $y$ are on the same side of $b_j$, opposite $z$. Since $y \lhd_g^s z$, then $x\lhd_g^s z$ too. This proves transitivity.
	
	The third claim is an immediate consequence of the first two.
\end{proof}

If $x_1,\ldots x_n$ are points of $[0,1]$, pairwise $g$-comparable, denote by $x_{(1)} < \ldots < x_{(n)}$ their order statistic (for the usual order on $[0,1]$). We then define \[\Perm_{g,s}(x_1,\ldots x_n) = \rank_{\lhd_g^s}(x_{(1)},\ldots,x_{(n)}).\]
Observe for instance \cref{BP_fig:eandf}. In this instance, $\Perm_{e,S}(t_1,\ldots,t_4) = (3214)$.

To understand the structure of these permutations, let us define the (signed) trees extracted from a (signed) excursion. Following Le Gall \cite{legall2005}, when $g$ is a CRT excursion and $t_1<\ldots<t_k$ are pairwise $g$-comparable\footnote{The definition there is stated differently and covers any continuous function $g$ and choice of points $t_1,\ldots t_k$}, the discrete plane tree with edge-lengths $\tau(g,t_1,\ldots, t_k)$ is constructed recursively as follows: 
\begin{itemize}
	\item If $k=1$, then $\tau(g,t_1)$ is a leaf labeled $t_1$.
	\item If $k\geq 2$, then the minimum of $g$ on $[t_1,t_k]$ is reached at a strict local minimum $b_i$ for some $i$, and there is $j \in \llbracket 2,k\rrbracket$ such that $\{t_1,\ldots t_{j-1}\} \subset (a_i,b_i)$, and $\{t_j,\ldots t_k\} \subset (b_i,c_i)$. Then $\tau(g,t_1,\ldots,t_k)$ is a root labeled $i$, spanning two subtrees $\tau(g,t_1,\ldots,t_{j-1})$ and $\tau(g,t_j,\ldots,t_k)$.
\end{itemize}
This yields a binary tree whose internal vertices are put in correspondence with branching points of $g$. Then, if $(g,s)$ is a signed excursion, we set $\tau^\pm(g,s,t_1,\ldots t_k)$ to be the tree $\tau(g,t_1,\ldots t_k)$, to which we add, at each internal node labeled $i$, the sign $s_i$. The following observation is capital: (recall the definition of $\perm$ from \cref{BP_sec:intro_separable})
\begin{observation} \label{BP_obs:permPerm}
	For any signed excursion $(g,s)$ and $g$-comparable $x_1,\ldots,x_n$, \[\Perm_{g,s}(x_1,\ldots x_n) = \perm(\tau^\pm(g,s,x_{(1)},\ldots,x_{(n)})).\]
\end{observation}

Going back to the example of \cref{BP_fig:eandf}, we see that $\tau^\pm(e,S,t_1,\ldots t_4)$ is the tree \includegraphics{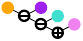}, whose associated permutation is indeed $(3214)$.

If $U_1, \ldots, U_k$ are independent uniform random variables in $[0,1]$, then they are almost surely pairwise $g$-comparable. We recall that the signed Brownian excursion $(e,S)$ is built by taking $e$ to be a normalized Brownian excursion, and $S$ an independent i.i.d. sequence of signs of bias $p$. Then a consequence of \cite[Theorem 2.11]{legall2005} is that the tree $\tau^\pm(e,S,U_{(1)},\ldots U_{(k)})$ is a uniform binary tree with $k$ leaves, independently decorated with i.i.d. signs of bias $p$. From \cref{BP_prob:bassino1} and \cref{BP_obs:permPerm} follows a new characterization of $\mu^p$, which we use in this paper.
\begin{proposition}\label{BP_prop:bassino}
	The permuton $\mu^p$ is determined by the relations
	\begin{equation}\label{BP_eq:caractmup}
	\forall k\geq 1,\quad \subperm_k(\mu^p) \stackrel d= \Perm_{e,S} (U_1,\ldots U_k).
	\end{equation}
\end{proposition}
\begin{remark}
	\label{BP_rk:cesar}
	This connection with the Brownian excursion was present in \cite{main} for $p=1/2$. The main result of that paper actually goes further: the conditional distribution of the l.h.s. given $\mu^{1/2}$ equals (in distribution) the conditional distribution of the r.h.s given $(e,S)$, jointly for all $k$ (see \cite[thm. 1.6]{main} and its proof). This indeed strongly hinted at the existence of a direct construction of $\mu^{1/2}$ from $(e,S)$, made explicit in the present paper.
\end{remark}

\section{The function $\varphi$}
\label{BP_sec:varphi}
\Cref{BP_thm:main} follows from the next two propositions.
\begin{proposition}
	If $g$ is a CRT excursion and $s$ a sequence of signs, then $(g,s,t) \mapsto \varphi_{g,s}(t)$ and $(g,s) \mapsto \mu_{g,s}$ are measurable.
	Furthermore, $\varphi_{g,s}{}_*\Leb = \Leb$, hence $\mu_{g,s}$ is a permuton.
\end{proposition}
\begin{proof}
	For the measurability, remark that $((g,s,t),u) \mapsto \idf[u\lhd_g^s t]$ is a measurable function, as a result of \cref{BP_mesenum3}.
	Then Fubini's theorem implies that its partial integral over $u$ is a measurable function of $(g,s,t)$.
	
	Now we only have to prove that $\varphi_*\Leb = \Leb$.
	Let $(Z_i)_{i\geq 1}$ be independent uniformly distributed random variables in $[0,1]$.
	For $k\geq 2$, let $U_{1,k} = \frac 1 {k-1} \#\{i\in\llbracket 2,k\rrbracket{  :} Z_i \lhd_g^s Z_1\}$ and $U_1 = \lim_{k\to\infty} U_{1,k}$.
	We can apply the law of large numbers conditionally on $Z_1$ to the sequence $\idf_{Z_2 \lhd^s_g Z_1}, \idf_{Z_3 \lhd_g^s Z_1}, \ldots$ (which is i.i.d given $Z_1$) to show that this limit is well defined and equal almost surely to $\Leb\{t :t\lhd_g^s Z_1\} = \varphi(Z_1)$.
	This means that $U_1$ has distribution $\varphi_*\Leb$.
	On the other hand, by exchangeability of the $Z_i$, the $U_{1,k}$ are uniform over $\{\frac 1{k-1}, \ldots, \frac {k-1}{k-1}\}$ so the distribution of the limit $U_1$ must be uniform.
	This means precisely that $\varphi_*\Leb = \Leb$.
\end{proof}

\begin{proposition}
	The Brownian separable permuton $\mu^p$ is distributed like $\mu_{e,S}$.
\end{proposition}
\begin{proof}
	By definition of $\mu_{g,s}$, $\subperm_k(\mu_{e,S})$ can be realized as $\rank(\mathbf Y) \circ \rank(\mathbf X)^{-1}$ where $X_1,\ldots X_k$ are independent uniform in $[0,1]$ and $Y_i = \varphi_{e,S}(X_i)$ for $i\in \llbracket 1,k\rrbracket$
	Since $x\lhd_e^{S} y$ implies $\varphi_{e,S}(x) \leq \varphi_{e,S}(y)$, and moreover since the $Y_i$ are almost surely distinct, then almost surely $\subperm_k(\mu_{e, S}) = \Perm_{e,S}(X_1,\ldots X_k)$.
	According to \cref{BP_prop:bassino}, this property characterizes $\mu^p$ among permutons.
\end{proof}

We now collect a few results about the excursion and the function $\varphi$.
The first one states that $[0,1]$ can almost be covered by a union of small subexcursions.
\begin{lemma}\label{BP_lem:coverbysubtrees}
	Let $g$ be a CRT excursion, and $\delta>0$, $\epsilon>0$.
	There exists a finite $I \subset \bb N$ such that the $([a_i,c_i])_{i\in I}$ are disjoint, $c_i-a_i\leq \epsilon$ for every $i$, and $\Leb(\bigsqcup_{i\in I} [a_i,c_i]) = \sum (c_i-a_i)>1-\delta$.
\end{lemma}
\begin{proof}
	Let $x$ be a leaf of the excursion $g$.
	Let $x_0<x$ be another leaf.
	Define recursively $b_{k_n}$ to be the most recent common ancestor of  $x_n$ and $x$, and $x_{n+1}$ to be a leaf in $(\max \{b_{k_n},x-\tfrac 1 n\},x)$.
	This is possible by density of the leaves.
	Then necessarily $x\in[a_{k_n}, c_{k_{n}}]$ and $a_{k_n}$ converges to $x$. 
	Hence $g(c_{k_n})=g(a_{k_n})$ converges to $g(x)$, which implies that $c_{k_n}-a_{k_n}$ converges to $0$ (otherwise $x$ couldn't be a leaf).
	Hence there must be a $i$ such that $|c_i-a_i|\leq \epsilon$ and $x\in[a_i,c_i]$.
	
	We deduce that $\bigcup_{i : c_i-a_i \leq \epsilon} [a_i,c_i]$ has measure 1.
	So a finite union can be found with measure $\geq 1-\delta$.
	Now thanks to \cref{BP_lem:genealogyofbp}, this union can be readily rewritten as a disjoint union.
\end{proof}

Now we want to characterize how the function $\varphi_{g,s}$ behaves on a pair of sibling subexcursions defined by an interval of the form $[a_i,c_i]$.
Set $a'_i=\varphi_{g,s}(a_i)$, $c'_i = a'_i +c_i - a_i$, $b'_i = a'_i + (b_i - a_i)\idf[s_i = \oplus] + (c_i - b_i)\idf[s_i = \ominus]$. 
The numbers $a'_i,b'_i,c'_i \in [0,1]$ can be interpreted as the equivalent of $a_i,b_i,c_i$ for the shuffled order. 
\begin{lemma}\label{BP_lem:localization}
	For $i\in \mathbb N$, we have
	\begin{align*}
	&\text{if }t\in [a_i,b_i]\text{ and }s_i = \oplus\text{, then }
	&\varphi_{g,s}(t) = a'_i + \Leb\{x\in [a_i,b_i] : x \lhd_g^s t\} \in [a'_i,b'_i].\\
	&\text{if }t\in [b_i,c_i]\text{ and }s_i = \oplus\text{, then } 
	&\varphi_{g,s}(t) = b'_i + \Leb\{x\in [b_i,c_i] : x \lhd_g^s t\} \in [b'_i,c'_i].\\
	&\text{if }t\in [a_i,b_i]\text{ and }s_i = \ominus\text{, then } 
	&\varphi_{g,s}(t) = b'_i + \Leb\{x\in [a_i,b_i] : x \lhd_g^s t\} \in [b'_i,c'_i].\\
	&\text{if }t\in [b_i,c_i]\text{ and }s_i = \ominus\text{, then } 
	&\varphi_{g,s}(t) = a'_i + \Leb\{x\in [b_i,c_i] : x \lhd_g^s t\} \in [a'_i,b'_i].
	\end{align*}
	If $t \in [0,a_i) \cup (c_i,1]$, then 
	\[
	\varphi_{g,s}(t) = \Leb\{x\in [0,a_i) \cup (c_i,1] : x \lhd_g^s t\}
	+ \idf[a_i\lhd_g^s t](c_i - a_i) \in [0,a'_i]\cup[c'_i,1] 
	\]
\end{lemma}
\begin{proof}
	We prove the first and last equalities, as the others have a symmetric proof.
	If $s_i = \oplus$, $t\in[a_i,b_i]$ and $u$ is a leaf, then $u\lhd^s_g t$ if and only if $u \in [0,a_i) \cup (c_i,1]$ and $u\lhd_g^s a_i$, or $u \in [a_i,b_i]$ and $u\lhd_g^s t$. The first claim follows by taking the measure of such $u$.
	
	For the last equality, we see that if $t\in[0,a_i)\cup(c_i,1]$ and $u\in[a_i,c_i]$, then $u\lhd^s_g t$ if and only if $a_i \lhd_g^s t$.
\end{proof}

\begin{lemma}\label{BP_lem:genealogyofbp2}
	If $[a_j,c_j] \subset (a_i,b_i)$, then either $s_i = \oplus$ and $[a'_j,c'_j] \subset [a'_i,b'_i]$, or $s_i = \ominus$ and $[a'_j,c'_j] \subset [b'_i,c'_i]$.
	
	If $[a_j,c_j] \subset (b_i,c_i)$, then either $s_i = \oplus$ and $[a'_j,c'_j] \subset [b'_i,c'_i]$, or $s_i = \ominus$ and $[a'_j,c'_j] \subset [a'_i,b'_i]$.
	
\end{lemma}
\begin{proof}
	The four claims have a symmetrical proof, hence we only prove the first.
	If $s_i=\oplus$ and $[a_j,c_j]\subset (a_i,b_i)$, then the previous lemma implies readily $a'_i \leq a'_j$.
	We need to prove $c'_j \leq b'_i$, that is $a'_j + c_j - a_j \leq a'_i + b_i - a_i$, which is equivalent to $a'_j - a'_i \leq a_j-a_i+b_i-c_j$.
	This is exactly the inequality of measures derived from the inclusion $\{x,a_i\lhd_g^s x\lhd_g^s a_j\} \subset [a_i,a_j]\sqcup[c_j,b_i]$
\end{proof}

Now we can prove \cref{BP_cor:cvoffunctions}.
\begin{proof}[Proof of \cref{BP_cor:cvoffunctions}] 
	We consider the  Kolmogorov distance between probability measures, which is the uniform distance on the bivariate CDFs ($d_{K}(\nu,\pi) = \sup_{0\leq x,y \leq 1}|\nu-\pi|([0,x]\times [0,y])$). 
	We use the fact that convergence of permutons is metrized by $d_{K}$ \cite[lemma 5.3]{hoppen}, and the following result: 
	\begin{lemma}\label{BP_lem:dks}
		If $\sigma \in \perms_n$, 
		$d_{K} (\mu_{\sigma}, (\Id,\varphi_{\sigma})_*\Leb) \leq \tfrac 2 n$
	\end{lemma}
	\begin{proof}
		It is enough to notice that both CDFs coincide on points whose coordinates are entire multiples of $1/n$ and use the fact that CDFs of permutons are $1$-Lipschitz \cite[eq. 7]{hoppen}
	\end{proof}
	All together, this implies $(\Id,\varphi_{\sigma_n})_*\Leb \xrightarrow d (\Id,\varphi_{e,S})_*\Leb$.
	With the Skorokhod coupling we can assume without loss of generality, that the convergence is in fact almost sure.
	Let $\epsilon$ and $\delta$ be positive real numbers, and apply \cref{BP_lem:coverbysubtrees}.
	Then		
	\begin{multline*}	
	\Leb(x:|\varphi_{\sigma_n}(x) - \varphi_{e,S}(x)|>\epsilon) 
	\leq \Leb(x:x\notin \bigsqcup_{i\in I} [a_i,c_i])\\
	+ \Leb(x : \exists i\text{ s.t. } x\in [a_i,c_i], \varphi_{\sigma_n}(x)\notin [a'_i,c'_i])
	\end{multline*}
	The first term is smaller than $\delta$ by construction, and the second term converges to $\Leb(x : \exists i\text{ s.t. } x\in [a_i,c_i], \varphi_{e,S}(x) \notin [a'_i,c'_i]) = 0$ because of the narrow convergence of $(\Id,\varphi_{\sigma_n})$ to $(\Id,\varphi_{e,S})$ and the Portmanteau theorem (indeed permutons put no mass on the boundary of rectangles, because they have uniform marginals).
	So for $q\geq1$,  $||\varphi_{\sigma_n} - \varphi_{e,S}||^q_{L^q}\leq  \epsilon^q + \delta + o(1)$.
	This last quantity can be made arbitrary small by choosing first $\epsilon$ and $\delta$ small enough and then $n$ large enough.
	We have proven almost sure convergence of $\varphi_{\sigma_n} \xrightarrow {L^p} \varphi_{e,S}$ in some coupling, hence the corollary. 
\end{proof}

We end this section by considering the following property of signed excursions $(g,s)$:
\begin{multline} \label{BP_eq:olt} \tag{A}
\forall i\neq j,\quad[a'_j,c'_j] \subset [a'_i,c'_i] \implies \{h_l : l\geq 1, [a'_l,c'_l]\subset[a'_i,c'_i]  \text{ and } [a'_j,c'_j]\subset [b'_l,c'_l]\} \\
%\text{ is dense in }[h_i,h_j]\\
%\\[a'_j,c'_j] \subset [a'_i,b'_i] \implies
\text{and } \{h_l : l\geq 1,  [a'_l,c'_l]\subset[a'_i,c'_i]  \text{ and } [a'_j,c'_j]\subset [a'_l,b'_l]\} \text{ are dense in }[h_i,h_j]
\end{multline}
It is very similar to the "order-leaf-tight" property of continuum trees defined in \cite{aldous1993}. Loosely said, it means that it is impossible to find a nontrivial ancestral path in the tree $\cali T_g$ without a density of points both on the right and on the left where a subtree is grafted. "left" and "right" are understood with regard to the shuffled order $\leq_g^s$. This is crucial to the proof of \cref{BP_thm:f}. We show that it holds almost surely in our setting.
\begin{proposition}
	\label{BP_olt}
	Let $g$ be a CRT excursion, $p \in (0,1)$ and $S$ be a random i.i.d. sequence of signs with bias $p$. Then with probability one, $(g,S)$ verifies property \eqref{BP_eq:olt}
	
\end{proposition}
\begin{proof}
	By symmetry we prove only the first claim and by countable union we fix $i$ and $j$. Let $K= \{ l \geq 1 : [a_l,c_l]\subset[a_i,c_i]  \text{ and } [a_j,c_j]\subset [b_l,c_l]\}$, and 
	$\widetilde K = \{l : l\geq 1, [a'_l,c'_l]\subset[a'_i,c'_i]  \text{ and } [a'_j,c'_j]\subset [b'_l,c'_l]\}$.
	For $y\in (h_i,h_j) \cap \bb Q$, consider $x = \sup\{t\in[a_i,a_j] : g(t) = y\}$. Then by definition $g(x) = y$ and $g(t)>y$ for $t>x$. Consider a sequence of leaves $x_n \nearrow x$ and the minimum $b_{k_n}$ of $g$ between $x_n$ and $a_i$. Then necessarily $k_n\in K$ and $x_n < b_{k_n} < x$. So $h_{k_n} \to y$.
	
	Now with probability one a subsequence $(k'_n)_n$ of $(k_n)_n$ can be found with $s_{k'_n} = \oplus$ for every $n$. Then \cref{BP_lem:genealogyofbp2} implies that $k'_n \in \widetilde K$, and $h_{k'_n} \to y$.
	By countable union over $y$ we have shown that $\overline{\{h_l, l\in \widetilde K\}}$ countains $(h_i,h_j) \cap \bb Q$. So it contains $[h_i,h_j]$ from which the proposition follows.
\end{proof}

An immediate consequence of property \eqref{BP_eq:olt} is the following improvement on \cref{BP_lem:genealogyofbp2}, with strict inclusions.
\begin{lemma}\label{BP_lem:genealogyofbp2mieux} Suppose $(g,s)$ verifies \eqref{BP_eq:olt}. Let $i\neq j$.
	
	If $[a_j,c_j] \subset (a_i,b_i)$, then either $s_i = \oplus$ and $[a'_j,c'_j] \subset (a'_i,b'_i)$, or $s_i = \ominus$ and $[a'_j,c'_j] \subset (b'_i,c'_i)$.
	
	If $[a_j,c_j] \subset (b_i,c_i)$, then either $s_i = \oplus$ and $[a'_j,c'_j] \subset (b'_i,c'_i)$, or $s_i = \ominus$ and $[a'_j,c'_j] \subset (a'_i,b'_i)$.
	
	If $[a_i,c_i]\cap [a_j,c_j] = \emptyset$, then $[a'_i,c'_i] \cap [a'_j,c'_j] = \emptyset$.	
\end{lemma}

%%%%%%%%%%%%%%%%%%%%%%%%%%%%%%%%%%%%%%%%%%%%%%%%%%%%%%%%%%%%%%%%%%%%%%%%%%%%%%%%

\section{The support of the permuton}
\label{BP_sec:onedim}

\Cref{BP_thm:dimension} follows readily from the two propositions of this section.

\begin{proposition}
	\label{BP_hausdorff1}
	For every signed excursion $(g,s)$, $\mu_{g,s}$ has Hausdorff dimension $1$ and its $1$-dimensional Hausdorff measure is $\leq \sqrt{2}$.
\end{proposition}

\begin{proof}
	Let $\pi_1, \pi_2$ denote the two coordinate projections of the unit square.
	For $U\subset [0,1]^2$, we write $\width(U) = \sup \pi_1(U) - \inf \pi_1(U)$ and $\height(U) = \sup \pi_2(U) - \inf \pi_2(U)$.
	
	We start by showing that $\dim_H(\supp(\mu))\geq 1$.
	If $\pi_1$ is the projection of the unit square to its first coordinate, then $\pi_1(\supp(\mu)) = [0,1]$, otherwise $\mu$ couldn't have a uniform marginal.
	We conclude with the following lemma, which is immediate from the definition of Hausdorff dimension:
	
	\begin{lemma}\label{BP_haussdimproj}
		If $\theta : (E,d_E) \to (F,d_F)$ is a contraction, then for $X\subset E$, $\dim_H(X)>\dim_H(\theta(X))$
	\end{lemma}
	
	To prove the upper bound, we apply \cref{BP_lem:coverbysubtrees} for some choice of $\epsilon > \delta >0$.
	Let $I$ be the set of indices provided by the lemma.
	Let $J = \{k: \exists i, j \in I , [a_i,c_i] \subset (a_k,b_k), [a_j,c_j] \subset (b_k,c_k)\}$.
	Let $K = I \sqcup J$
	We have the following fact, which is a direct consequence of the nested structure of the $[a_i,c_i]$.
	\begin{fact}\label{BP_fct:iril}
		For every $i \in J$, there exists an $i_l\in K$ such that for every $j\in K$, $[a_j,c_j] \subset [a_i,b_i]$ implies $[a_j,c_j] \subset[a_{i_l},c_{i_l}]\subset [a_i,b_i]$.
		Similarly for every $i \in J$, there exists be an $i_r\in K$ such that for every $j\in K$, $[a_j,c_j] \subset [b_i,c_i]$ implies $[a_j,c_j] \subset[a_{i_r},c_{i_r}]\subset [b_i,c_i]$.
		Also there exists $\star \in J$ such that for every $k\in K$, $[a_k, c_k] \subset [a_\star, c_\star]$.
	\end{fact}		
	We can define the following subsets of the unit square, which we use to cover $\supp(\mu_{g,s})$:		
	\begin{align*}
	A_i &= ([a_i,a_{i_l}] \cup [c_{i_l},b_{i}])\times([a'_i,a'_{i_l}]
	\cup  [c'_{i_l},b'_{i}])\\
	& \cup ([b_i,a_{i_r}] \cup [c_{i_r},c_{i}])\times([b'_i,a'_{i_r}]
	\cup  [c'_{i_r},c'_{i}]) 
	& \text{ if }i \in J\text{ and }s_i = \oplus  \\
	A_i &= ([a_i,a_{i_l}] \cup [c_{i_l},b_{i}])\times([b'_i,a'_{i_l}]
	\cup  [c'_{i_l},c'_{i}])\\
	& \cup ([b_i,a_{i_r}] \cup [c_{i_r},c_{i}])\times([a'_i,a'_{i_r}]
	\cup  [c'_{i_r},b'_{i}]) 
	& \text{ if }i \in J\text{ and }s_i = \ominus \\
	A_i & = [a_i,c_i]\times[a'_i,c'_i] & \text{ if }i \in I\\
	A_0 &=([0,a_\star] \cup  [c_{\star},1]) \times ([0,a'_\star] \cup  [c'_{\star},1])
	\end{align*}
	By construction and \cref{BP_fct:iril}, $\bigcup_{i\in K \cup \{0\}} \pi_1(A_i) = [0,1]$, and \cref{BP_lem:localization} implies that for $x \in \pi_1(A_i)$, $(x,\varphi_{g,s}(x)) \in A_i$. 
	This one has:
	\begin{equation}
	\label{BP_cover}
	(\Id,\varphi_{g,s})[0,1]\subset  \bigcup_{i\in K \cup \{0\}} A_i.
	\end{equation}
	\begin{figure}[htb]
		\centering
		\includegraphics{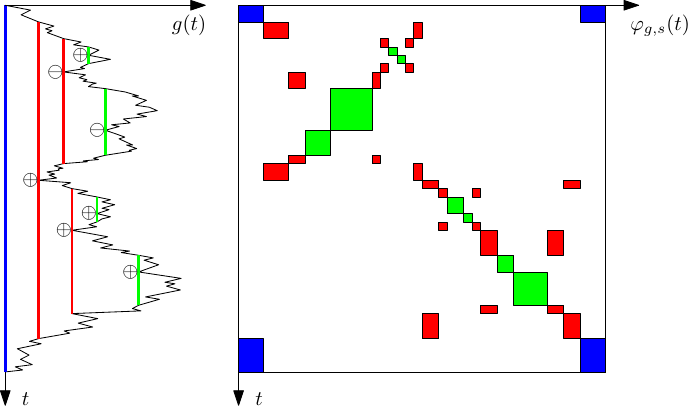}
		\caption{$A_0$ in blue, $A_i$ for $i\in I$ in green, and $A_i$ for $i\in J$ in red.}
	\end{figure}
	The rest of the proof is devoted to rewriting the right-hand side of \eqref{BP_cover} as an union of sets in which we control the sum of diameters.
	Now, for $i \in I$, $\diam(A_i) = \diam([a_i,c_i]\times[a'_i,c'_i]) = \sqrt{2}(c_i-a_i)$.
	We deduce that
	\begin{equation}\label{BP_eqn:diamA_i}
	\sum_{i\in I} \diam(A_i) \leq \sqrt 2. 
	\end{equation}
	For $i\in J$, %\cref{BP_lem:genealogyofbp2} tells us that 
	$A_i$ is the union of 8 rectangles $A_i^1,\ldots A_i^8$. %2[(c_i-a_i)-(c_{i_l}-a_{i_l}) - (c_{i_r}-a_{i_r})] = 
	We have that
	\begin{align*}
	\sum_{j=1}^8 \width(A_i^j) &= 2[(c_i-a_i)-(c_{i_l}-a_{i_l}) - (c_{i_r}-a_{i_r})]\\
	\sum_{j=1}^8 \height (A_i^j)&= 2[(c'_i-a'_i)-(c'_{i_l}-a'_{i_l}) - (c'_{i_r}-a'_{i_r})].
	\end{align*}
	And both these quantities are equal and their value is 2 $\Leb(\pi_1(A_i))$.
	Similarly, $A_0$ is the union of 4 rectangles $A_0^1,\ldots, A_0^4$ whose widths and heights both sum to $2\Leb(\pi_1(A_0))$. Hence	
	\begin{align}\label{BP_eqn:diamA_j}
	\sum_{j=1}^4 \diam(A_0^j) + \sum_{i\in J} \sum_{j=1}^8\diam(A_i^j) 
	&\leq \sum_{j=1}^4 (\width + \height)(A_0^j)  
	+ \sum_{i\in J} \sum_{j=1}^8(\width + \height)(A_i^j) \nonumber\\
	&= 4\Leb(\pi_1(A_0))
	+4 \sum_{i\in J} 
	\Leb(\pi_1(A_i)) \nonumber\\
	&= 4 \Leb([0,1]\setminus \bigcup_{i\in I}[a_i,c_i]) 
	\leq 4\delta
	\end{align}
	By taking the closure and rewriting the right-hand side in \cref{BP_cover}, we get 
	\begin{equation}\label{BP_eq:the_cover}
	\supp (\mu_{g,s}) 
	\subset \overline{(\Id,\varphi_{g,s})[0,1]}
	\subset \left(\bigcup_{i\in I} A_i \right)
	\cup \left(\bigcup_{j=0}^4 A_0^j\right)
	\cup \left(\bigcup_{i\in J} \bigcup_{j=1}^8 A_i^j  \right)
	\end{equation}
	Summing \eqref{BP_eqn:diamA_i} and \eqref{BP_eqn:diamA_j} shows that the sum of diameters in the cover \eqref{BP_eq:the_cover} can't exceed $4\delta + \sqrt 2$.
	Moreover, each square and rectangle in the cover has diameter bounded by $\sqrt{2}\epsilon$.
	This implies that $\supp(\mu)$ has 1-dimensional Hausdorff measure bounded above by $\sqrt 2$.
\end{proof}

\begin{proposition}
	If $S$ is an i.i.d sequence of nondeterministic signs, then $\supp(\mu_{g,S})$ is almost surely totally disconnected.
\end{proposition}
\begin{proof}
	We re-use the notations of the last proof, with $\epsilon > \delta > 0$. We now show that almost surely, we can build sets $\bar I\supset I$ and $\bar J\supset J$ such that 
	\begin{enumerate}
		\item the statement of \cref{BP_fct:iril} is still true when $J$ is replaced by $\bar J$ and $K$ by $\bar K = \bar I \sqcup \bar J$,
		\item for all $i\in \bar I$, $c_i-a_i \leq \epsilon$,
		\item $\Leb([0,1] \setminus \bigsqcup_{i\in \bar I}  [a_i,c_i]) <\delta$,
	\end{enumerate}
	with the following added constraint:
	\begin{equation}\label{BP_eqn:alternate}
	\forall	i\in J,\quad s(b_{i_r}) =s(b_{i_l}) \neq s(b_i).
	\end{equation}
	This is done by adding successively indices to $I$ in order to create new branching points in between two branching points of the same sign.
	Condsider $i\in J$ and its left child $i_l$, with $s_i= s_{i_l}=\epsilon$.
	We can build, as in the proof of \cref{BP_lem:coverbysubtrees}, an infinite sequence  $(b_{r_n})_n$ such that $[a_{r_n}, c_{r_n}] \subset [a_i, b_i]$ and $[b_{r_k},c_{r_k}] \supset [a_{i_l}, c_{i_l}]$. 
	Almost surely, one of the $r_n$, which we denote $j=j(i,i_l)$, is such that $s_j\neq \epsilon$. 
	We can then find, by the same reasoning, a $k=k(j(i,i_l))$ such that $[a_k,c_k] \subset [a_j,c_j]$ and $s_k = \epsilon$. We proceed similarly for every $i\in J$ such that $s_i= s_{i_r}$. We can now set 
	\begin{align*}\bar I &= I \cup \{k(i,i_l) : i\in J, s_i \neq s_{i_l}\}\cup \{k(i,i_r) : i\in J, s_i \neq s_{i_r}\}\\
	\bar J &= J \cup \{j(i,i_l) : i\in J, s_i \neq s_{i_l}\}\cup \{j(i,i_r) : i\in J, s_i \neq s_{i_r}\}.
	\end{align*}
	By construction, \cref{BP_fct:iril} applies to $\bar I$ and $\bar J$, and \eqref{BP_eqn:alternate} is verified.
	
	Now we can define the sets $(A_i)_{i\in \bar K \cup \{0\}}$ as in the previous proof, and we still have 
	\begin{equation*}
	\supp \mu_{e, S} \subset C = \bigcup_{i\in \bar K \cup \{0\}} A_i.
	\end{equation*}
	We will show that the diameter of any connected component of C is almost surely bounded by $4\epsilon + 2\delta$. 
	This is enough to show that $\supp(\mu_{g,S})$ is totally disconnected.
	
	For $x\in C$, let us denote by $\cali C(x)$ the connected component of $C$ containing $x$, and for $X\subset C$, set $\cali C(X) = \cup_{x\in X}\cali C(x)$. We now set, for $i\in \bar I$, $B_i = \cali C(A_i)$, for $i\in \bar J$ $B_i = \cali C(A_i) \setminus \cali C(A_{i_l}) \setminus \cali C(A_{i_r})$, and $B_0 = \cali C(A_0) \setminus \cali C(A_*)$. Then, immediate induction yields
	\begin{equation*}
	C = \bigsqcup_{i\in \bar K \cup \{0\}} B_i.
	\end{equation*}
	Now remark that the sets $B_i$ were obtained by inclusion and exclusion of full connected components of $C$. Hence each connected component of $C$ appears as a connected component of one of the $B_i$, that we now consider.
	
	It turns out (see \cref{BP_fig:disc}) that for $i\in \bar I$, $B_i$ has only one connected component, and its diameter is bounded above by $4\epsilon+2\delta$.
	For $i\in \bar J$, $B_i$ has three connected components, whose diameter is bounded above by $2\delta$. For $i=0$, $B_0$ has two connected components, and their diameter is also bounded above by $2\delta$.
	\begin{figure}[htb]
		\centering
			\includegraphics{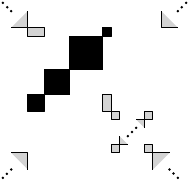}\quad\quad\quad
			\includegraphics{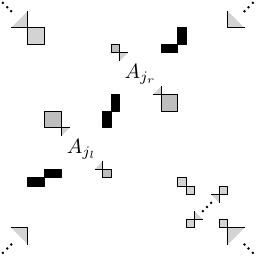}\quad\quad\quad
			\includegraphics{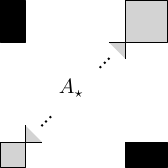}
		\caption{Left: $B_i$ for $i\in I$, in the case $s_i = \oplus$, $i=j_l$ for some $j$. Center: $B_j$ for $j\in J$, in the case $s_j = \oplus$, $j=j'_l$ for some $j'$. Right: $B_0$, in the case $s_\star = \oplus$.\label{BP_fig:disc}}
	\end{figure} 
\end{proof}

\section{Self-similarity}
\label{BP_sec:ss}
Given a CRT excursion $g$ and one of its branching points $b$ , one can build three subexcursions by cut-and-pasting, which encode the three connected components of $\cali T_g\setminus \{p_g(b)\}$. The goal of this section is do the same procedure on signed excursions, and observe the consequences on the associated permutons. This will allow us to prove \cref{BP_thm:ss} in a "reversed" fashion: we start from $\mu$, build $\mu_1$, $\mu_2$ and $\mu_3$ by cutting along a suitably chosen branching point, as to be able to use a result of Aldous \cite{aldous1994} and identify the distribution and relative sizes of the subexcursions.

Let $(g,s)$ be a signed excursion. Given $\ibar\in \bb N$, we can obtain 3 excursions by looking at the values of $g$ on $[a_{\ibar},b_\ibar]$, $[b_\ibar,c_\ibar]$ and $[0,a_\ibar] \sqcup [c_\ibar,1]$. More precisely, following \cite{aldous1994}, we define
\begin{equation}\label{BP_eq:definitiondeltas}
\Delta_0 = 1-c_\ibar +a_\ibar, \Delta_1 = b_\ibar-a_\ibar, \Delta_2 = c_\ibar - b_\ibar, X_0 =\frac  {a_\ibar}{\Delta_0}, Y_0 = \frac {a'_\ibar}{\Delta_0}, \beta = s_\ibar.
\end{equation}
Given these constants, we can define the contractions $\theta_k,\eta_k,\zeta_k$ for $k\in\{0,1,2\}$, as in \eqref{BP_eq:definitiontheta}, and
\begin{equation}
g_k = \frac {1}{\sqrt{\Delta_k}}g \circ \eta_k, \quad k\in\{0,1,2\}.
\end{equation}

Because each $\eta_k$ is a piecewise affine function, it pulls back the strict local minima of $g$ that are in the interior of $\Image(\eta_k)$ onto strict local minima of $g_k$. This is made explicit in the following result:
\begin{proposition}\label{BP_prop:bourbaki}
	For $k\in \{0,1,2\}$, there is an injective map $\vartheta_k : \bb N \to \bb N$, such that
	\[\forall i \in \bb N,\quad\eta_k(b_i(g_k)) = b_{\vartheta_k(i)}(g) . \]
	Moreover, the $\vartheta_k(\bb N)$, for $k\in \{0,1,2\}$, form a partition of  $\bb N \setminus \{\ibar\}$.
	Finally, for $k\in \{0,1,2\}$, the map $(g,\ibar,i) \mapsto \vartheta_k(i)$ is measurable.
	\begin{proof}
		We set $\vartheta_k(i) = \min \{j \in \bb N : \eta_k(b_i(g_k)) = b_{j}(g)\}$, and the measurability claim follows from measurability of $(i,g) \mapsto b_i(g)$, $(\ibar,g) \mapsto \eta_k$ and $(\ibar,g) \mapsto g_k$. The other claims are immediate by construction and from the definition of a measurable enumeration.
	\end{proof}
\end{proposition}

We can now transport the signs of $g$ onto signs of the $g_k$ by setting $s^k_i = s_{\vartheta_k(i)}$ for $k \in \{0,1,2\}$ and $i\in \bb N$. A result of this construction is the following crucial observations:
\begin{observation}
	For $x<y\in[0,1]$, and $k\in \{0,1,2\}$, $x\lhd_{g_k}^{s } y$ if and only if $\eta_k(x)\lhd_g^s\eta_k(y)$.
\end{observation}
\begin{observation}\label{BP_obs:sk_measurable}
	The map $(g,\ibar,(s_i)_{i\in \bb N}) \mapsto s^k_i$ is measurable for every $i\in \bb N$ and $k\in \{0,1,2\}$
\end{observation}

Now we want to use \cref{BP_lem:localization} to show that our function $\varphi_{g,s}$ can be cut out into rescaled copies of $\varphi_{g_k,s^k}$, which translates immediately in termes of measures.
\begin{proposition}\label{BP_cor:rescaledcopies}
	For $\ibar \in \bb N$, $k\in\{0,1,2\}$ and $t\in[0,1]$,
	\begin{equation}\label{BP_eqn:selfsim}
	\varphi_{g,s} \circ \eta_k(t) = \zeta_k \circ \varphi_{g_k,s^k}(t).
	\end{equation}
	As a consequence,
	\[\mu_{g,s} = \sum_{k=0}^2 \Delta_k \cdot (\theta_k{}_* \mu_{g_k,s^k}). \]
\end{proposition}

\begin{proof}
	Let us prove \eqref{BP_eqn:selfsim} for $k=0$.
	\begin{align*}
	\varphi_{g,s}(\eta_0(t)) 
	&= \Leb\{x\in [0,a_\ibar) \cup (c_\ibar,1] : x \lhd_g^s \eta_0(t)\} 
	+ \idf[a_\ibar\lhd_g^s t](c_\ibar - a_\ibar)  \\
	&= \Leb\{x\in [0,a_\ibar) \cup (c_\ibar,1] : x \lhd_g^s \eta_0(t)\}\\ 
	&\quad+(c_\ibar - a_\ibar)\idf\left[\Leb\{x\in [0,a_\ibar) \cup (c_\ibar,1] : x \lhd_g^s \eta_0(t)\} 
	> a'_\ibar \right]  \\
	&=\Delta_0\Leb\{y\in [0,1] : y \lhd_{e_0}^{s^0} t\} + (1-\Delta_0)\idf\left[ \Delta_0\Leb\{y\in [0,1] : y \lhd_{e_0}^{s^0} t\} > \Delta_0 Y_0\right]\\
	&=\zeta_0(\varphi_{e_0,s^0}(t)) 
	\end{align*}
	Where the first two equalities come from \cref{BP_lem:localization} and the third is the result of the change of variable $x=\eta_0(y)$. Now, for $k=1$, 
	\begin{align*}
	\varphi_{g,s}(\eta_1(t)) 
	&= a'_\ibar + (b'_\ibar-a'_\ibar)\idf[s_\ibar = \ominus] + \Leb\{x\in [a_\ibar,b_\ibar] : x \lhd_g^s \eta_1(t)\} \\
	%= \Leb\{x\in [0,a_\ibar) \cup (c_\ibar,1) : x \leq_e^S t\}+ (c_\ibar-b_\ibar)\idf[s_\ibar = \ominus]+  (b_\ibar-a_\ibar)\Leb\{x\in [0,1] : x \leq_{e_1}^{S_1} t\} \\
	&= \Delta_0Y_0 +  \Delta_2\beta + (b_\ibar-a_\ibar)\Leb\{y\in [0,1] : y \lhd_{g_1}^{s^1} t\} \\
	&=\zeta_1(\varphi_{g_1,s^1}(t))
	\end{align*}
	where the first equality comes from \cref{BP_lem:localization} and the second is the result of the change of variable $x=\eta_1(y)$. The case $k=2$ is similar.
\end{proof}

This is all we need to show \cref{BP_thm:ss}.
\begin{proof}[Proof of \cref{BP_thm:ss}]
	If $e$ is an Brownian excursion, and $X_l<X_r$ are reordered uniform independent random variables in $[0,1]$, independent of $e$, then almost surely there is a $\ibar$ such that $b_\ibar = \argmin_{[X_l,X_r]} e$.
	Define  $\Delta_0, \Delta_1, \Delta_2,X_0,Y_0,\beta$ as in \eqref{BP_eq:definitiondeltas}.
	This allows us to define the $\theta_k$ as in \eqref{BP_eq:definitiontheta} and the $e_k,s^k$ as before.
	
	A result of Aldous \cite[cor. 5]{aldous1994} states that  $e_0, e_1,e_2$ are Brownian excursions, $(\Delta_0,\Delta_1,\Delta_2)$ is a $\Dirichlet(\tfrac 1 2,\tfrac 1 2,\tfrac 1 2)$ partition of 1, and $X_0$
	%U_1,U_2$ are uniform variables
	is uniform in $[0,1]$, all these random variables being independent.
	
	Now, as a consequence of \cref{BP_obs:sk_measurable}, for $k\in[0,1]$ and $i \in \bb N$, $S^k_i$ is a random variable.  Given $e$ and $\ibar$, the $S^k$ for $k\in[0,1]$ and $\beta$ are permutations of disjoint subsequences of $S$. As a result, the $S^k$ and $\beta$ are independent (and independent of $(e,X_l,X_r)$), and distributed as i.i.d. sequences of signs of bias $p$.
	
	We finally set $\mu_k = \mu_{e_k,S^k}$ for $k\in\{0,1,2\}$ and need only prove
	\begin{equation}\label{BP_eq:Y_0}
	Y_0 = \varphi_{e_0,S^0}(X_0) \text{ a.s.}
	\end{equation}
	to show that the collection of random variables $((\Delta_k)_{k\in\{0,1,2\}}, (\mu_k)_{k\in\{0,1,2\}},  (X_0,Y_0), \beta)$ has the joint distribution assumed in \cref{BP_thm:ss}.  \Cref{BP_cor:rescaledcopies} then yields the theorem.
	
	Let us now prove \eqref{BP_eq:Y_0}.
	\begin{align*}
	\Delta_0 Y_0 = a'_i = \Leb\{x\in [0,a_i) \cup (c_i,1) : x \leq_e^S a_i\}
	&= \Delta_0\Leb\{y\in [0,1] : y \leq_{e_0}^{S_0} \eta_0^{-1} (a_i)\}\\
	&= \Delta_0\Leb\{y\in [0,1] : y \leq_{e_0}^{S_0} X_0\}\\
	&= \Delta_0 \varphi_{e_0,S^0}(X_0). \qedhere
	\end{align*}
\end{proof}

\begin{remark}\label{BP_rk:ag}
	As seen in the proof, \cref{BP_thm:ss} is a direct consequence of the self-similarity property of the Brownian CRT \cite[thm. 2]{aldous1994}. 
	It was shown \cite{albenque2015} that this property actually characterizes the Brownian CRT in the space of measured $\bb R$-trees. 
	We believe that the arguments of Albenque and Goldschmidt can be transposed in our setting, to show that the law of $\mu^p$ is the only distribution on permutons which verifies  \eqref{BP_eq:ss}. 
	The main reason backing that claim is the following: permutons are characterized by their finite-dimensional marginals, just like measured $\bb R$-trees are determined by their reduced trees (see section 3 in \cite{albenque2015}).
\end{remark}

\section{Expectation of the permuton}
\label{BP_sec:expectation}
In this section 
we shall compute the density function of the averaged permuton $\E\mu^p$ for $p\in(0,1)$. We know that $\mu^p = \mu_{e,S}$, where $e$ is a normalized Brownian excursion and $S$ is an independent sequence of i.i.d. signs with bias $p$.
Since for fixed $(g,s)$, the measure $\mu_{g,s}$ is the distribution of the random pair $(U,\varphi_{g,s}(U))$ with $U$ uniform in $[0,1]$, then by Fubini's theorem, we get the following:
\begin{lemma}
	\label{BP_lem:expectation1}
	$\E\mu^p$ is the distribution of the random pair $(U, \varphi_{e,S}(U))$, where $e$ is a normalized Brownian excursion, $S$ is an independent sequence of i.i.d. signs with bias $p$, and $U$ is uniform, those three random variables being independent.
\end{lemma} 

Let $(B_t)_{0\leq t\leq 1}$ be a normalized Brownian bridge between $0$ and $0$.
Define its local time at $0$ as follows: for $t\in [0,1]$, set $L_t = \lim_{\varepsilon \to 0} \frac 1 {2\varepsilon} \int_0^t \idf_{0\leq |B_s| \leq \varepsilon}ds$ in probability.
Define also its right-continuous inverse $(T_l)_{l\geq 0}$.
We set  $\Delta T_l =  T_l - T_{l^-}$ for $l\geq 0$.
We suppose that each $l\geq 0$ such that $\Delta T_l >0$ is equipped with an independent sign $\epsilon_l$ with bias $p$.
We will use a result of Bertoin and Pitman \cite{BertoinPitman} to rewrite the measure $\E\mu^p$ as the distribution of some functional of $B$.
\begin{lemma}\label{BP_lem:expectation2}
	The measure $\E\mu^p$ is the distribution of $\left(\frac{ P_1+P_2}{P_1+P_2+P_3+P_4}, \frac{ P_1+P_4}{P_1+P_2+P_3+P_4}\right)$, where 
	\begin{equation}
	\label{BP_eq:P_1to4}
	\begin{gathered}
	P_1 = \textstyle{\sum_{l<L_1/2, \epsilon_l = \oplus} \Delta T_l, \quad P_2 = \sum_{l<L_1/2, \epsilon_l = \ominus} \Delta T_l} \\
	\textstyle{P_3 = \sum_{l>L_1/2, \epsilon_l = \oplus} \Delta T_l, \quad P_4 = \sum_{l>L_1/2, \epsilon_l = \ominus} \Delta T_l}
	\end{gathered}
	\end{equation}
\end{lemma}
\begin{proof}
	We will build a suitable coupling of $(e,S,U)$ on one hand, and $(B,\epsilon)$ on the other hand. Start with the bridge $B$, and set $U = T_{L_1/2}$. Define $(K_t)_{0\leq t\leq 1}$ as follows: $K_t = L_t$ for $0\leq t\leq U$ and $K_t = L_1 - L_t$ when $U\leq t\leq 1$. Theorem 3.2 of \cite{BertoinPitman} tells us that if we set $e = K+\lvert B \rvert$, then $(e,U)$ is distributed as a Brownian excursion with an independent uniform variable in $[0,1]$. Moreover, the following holds almost surely: for $0\leq t \leq U$, $K_t = \inf_{t\leq s \leq U} e_s$ and for $U\leq t\leq 1$, $K_t = \inf_{U\leq s \leq t} e_s$. Finally let $S$ be a sequence of i.i.d. signs with bias $p$, independent of $(B,e,U)$. The triple $(e,S,U)$ has the desired distribution. We can transfer some of the signs of $S$ to form the marking process $(\epsilon_l)_{l\geq 0, \Delta_l>0}$. First remark that almost surely, $U$ is not a one-sided local minimum of $e$. For $l\geq 0$ such that $\Delta T_l>0$,
	\begin{itemize}
		\item either $l<L_1/2$ and then $T_{l^-}<T_l<U$, in which case $T_l$ is an inner local minimum $b_{\imath_l}$ of $e$ for some $\imath_l \in \bb N$. We then set $\epsilon_l = S_{\imath_l}$.
		\item either $l>L_1/2$ and then $T_{l^-}<T_l<U$, in which case $T_{l^-}$ is an inner local minimum $b_{\imath_l}$ of $e$ for some $\imath_l \in \bb N$. We then set $\epsilon_l = S_{\imath_l}$.
	\end{itemize}
	The sequence $(\imath_l)_{l:\Delta T_l >0}$ is a random injection into $\mathbb N$ that solely depends on $B$. So conditional on $B$, the signs in $(\epsilon_l)_{l:\Delta T_l >0}$ are i.i.d. and of bias $p$. Then $(B,\epsilon)$ has the desired distribution.	
	
	We now show that in this coupling we have the almost sure equality $(U, \varphi_{e,S}(U)) = \left(\frac{ P_1+P_2}{P_1+P_2+P_3+P_4}, \frac{ P_1+P_4}{P_1+P_2+P_3+P_4}\right)$. Then \cref{BP_lem:expectation1} implies the present lemma. If we define 
	\begin{equation*}
	\begin{gathered}
	\textstyle{\hat P_1 = \Leb\{t : 0\leq t \leq U,t  \lhd_e^S U \},\quad \hat P_2 = \Leb\{t : 0\leq t \leq U,t  \rhd_e^S U \}}, \\ 
	\textstyle{\hat P_3 = \Leb\{t : U\leq t \leq 1,t \rhd_e^S U \}, \quad \hat P_4 = \Leb\{t : U\leq t \leq 1,t  \lhd_e^S U \}},
	\end{gathered}
	\end{equation*}
	then it is immediate that almost surely, $\hat P_1+\hat P_2+\hat P_3+\hat P_4 = 1$,  $\hat P_1 +\hat P_2 = U$ and $\hat P_1+\hat P_4 = \varphi_{e,S}(U)$.
	Now we need only show that the $P_i = \hat P_i$ for $1\leq i\leq 4$. For instance for $i=1$, we need to observe that $t\in [0,1]$ is such that $t<U$ and $t\lhd_e^S U$ if and only if there is a $b_i\in (t,U)$ such that $b_i$ is the unique minimum of $e$ on $[t,U]$ and $S_i = \oplus$. Such $b_i$ is necessarily equal to $T_l$ for some $l<\ L_1/2$ such that $T_{l^-} < t < T_l$, and then $S_i = \epsilon_l$. We have shown the following logical equivalence for $t\in [0,1]$:
	\begin{equation*}
	t \leq U \text { and }t\lhd_e^S U \iff \exists\, l<L_1/2 \text{ s.t. } T_{l^-} < t < T_l \text { and } \epsilon_l = \oplus.
	\end{equation*}
	Taking the Lebesgue measure on both sides yields $\hat P_1 = P_1$. For $i=2,3,4$, the proof is symmetric.
\end{proof}

Let $\cali U$ be the set of continuous excursions of variable length, with $R:\cali U \to \bb R^+$ denoting the length statistic. Let $N$ be the It\=o excursion measure of Brownian motion. For $\theta \geq 0$, define the measure $\Lambda^\theta(dr) = e^{-\theta r} N(R\in dr)$.
Denote by $(X_l^\theta)_{l\geq 0}$ the process of sums up to time $l$ of a Poisson point process of intensity $dt\Lambda^\theta$. This is a well-defined process because $\int \Lambda^\theta(dr)(r\wedge 1)$ is finite. We can state the following rewriting of the distribution $\E\mu^p$.
\begin{lemma}\label{BP_lem:expectation3}
	For any $\theta>0$, $\E\mu^p$ is the distribution of $\left(\frac{\cali P_1+\cali P_2} {\cali P_1+ \cali P_2+\cali P_3+\cali P_4}, \frac{ \cali P_1+\cali P_4}{\cali P_1+\cali P_2+\cali P_3+\cali P_4}\right)$, where conditional on a random variable $\lambda_Y$ with exponential distribution of parameter $\sqrt{2\theta}$, we define the variables $\cali P_1$, $\cali P_2$, $\cali P_3$ and $\cali P_4$ to be independent with $\cali P_1 \stackrel d = \cali P_3 \stackrel d = X^\theta_{p\lambda_Y/2}$ and $\cali P_2 \stackrel d = \cali P_4 \stackrel d = X^\theta_{(1-p)\lambda_Y/2}$.
\end{lemma}
\begin{proof}
	Let us reuse the notations of \cref{BP_lem:expectation2}. We make use of the results of Perman and Wellner \cite{PermanWellner}, which show that the most tractable object in terms of its excursions is not the normalized Brownian bridge, but the random-length bridge $(\beta_t)_{t\geq 0}$ defined as follows:  $\beta_t = \idf_{0\leq t \leq Y}\sqrt{Y}B_{t/Y}$ where $Y$ is a random variable of distribution $\Gamma(1/2,\theta)$ independent of $B$. Its local time $\lambda$, inverse local time $\tau$ and jump process $\Delta \tau$  are related to those of $B$ by $\lambda_t = \sqrt Y L_{t/Y}$, $\tau_l = Y T_{l/\sqrt Y}$ and $\Delta \tau_l = Y \Delta T_{l/\sqrt Y}$. The marking process $\epsilon$ can be modified accordingly by setting $\varepsilon_l = \epsilon_{l/\sqrt{Y}}$ for $l\geq 0$ such that $\Delta \tau_l >0$.
	
	Now if we set 
	\begin{equation*}
	\begin{gathered}
	\cali P_1 = \textstyle{\sum_{l<\lambda_1/2, \epsilon_l = \oplus} \Delta \tau_l, \quad \cali P_2 = \sum_{l<\lambda_1/2, \epsilon_l = \ominus} \Delta \tau_l} \\
	\textstyle{\cali P_3 = \sum_{l>\lambda_1/2, \epsilon_l = \oplus} \Delta \tau_l, \quad \cali P_4 = \sum_{l>\lambda_1/2, \epsilon_l = \ominus} \Delta \tau_l}
	\end{gathered}
	\end{equation*}
	then by construction, $\left(\frac{ P_1+P_2}{P_1+P_2+P_3+P_4}, \frac{ P_1+P_4}{P_1+P_2+P_3+P_4}\right) = \left(\frac{\cali P_1+\cali P_2} {\cali P_1+ \cali P_2+\cali P_3+\cali P_4}, \frac{ \cali P_1+\cali P_4}{\cali P_1+\cali P_2+\cali P_3+\cali P_4}\right)$.
	
	We now have to identify the joint distribution of the $\cali P_i$. It results from \cite[thm 1 and 4]{PermanWellner} that $\lambda_Y$ is distributed as an exponential random variable of parameter $\sqrt{2\theta}$, and that, conditional on $\lambda_Y$, the excursions of $\beta$ away from $0$, parametrized by the local time, form a Poisson point process of intensity $dle^{-\theta R(w)} N(dw)$ over $[0,\lambda_Y]\times \cali U$. 
	The random set $\{(l,\Delta \tau_l), l\geq 0, \Delta_l>0\}$, which is just the point process of excursion lengths, is then also Poisson with intensity $dl \Lambda^\theta(dt)$ over $[0,\lambda_Y]\times \bb R_+$. This results from the mapping property of Poisson processes.
	Now, since the marking process $(\varepsilon_l)_{l\geq 0}$ is a choice of i.i.d. marks, chosen independent of $B$, the marking property of point processes \cite[sect. 2.3]{Kingman} tells us that $\{(l,\Delta \tau_l, \varepsilon_l), l\geq 0, \Delta_l>0\}$ is itself a Poisson process of intensity $dl \Lambda^\theta(dt)(p\delta_\oplus + (1-p) \delta_\ominus)(d\varepsilon)$ over $[0,\lambda_Y]\times \bb R_+\times \{\oplus,\ominus\}$.
	
	Since they are functionals of the same Poisson process restricted to disjoint subsets, the processes $\{\Delta \tau_l, 0\leq l \leq \lambda_Y/2, \Delta_l>0, \varepsilon_l = \oplus\}$, $\{\Delta \tau_l, 0\leq l \leq \lambda_Y/2, \Delta_l>0, \varepsilon_l = \ominus\}$, $\{\Delta \tau_l,\lambda_Y/2\leq l \leq \lambda_Y, \Delta_l>0, \varepsilon_l = \oplus\}$ and $\{\Delta \tau_l,\lambda_Y/2\leq l \leq \lambda_Y, \Delta_l>0, \varepsilon_l = \ominus\}$, are independent.
	Moreover, by the mapping property, they are themselves Poisson, with respective intensity measures $\frac {p\lambda_{Y}}2\Lambda_\theta(dr)$, $\frac{(1-p)\lambda_{Y}}2\Lambda_\theta(dr)$, $\frac {p\lambda_{Y}}2 \Lambda_\theta(dr)$ and $\frac {(1-p)\lambda_{Y}}2 \Lambda_\theta(dr)$.
	The lemma follows.
\end{proof}

\begin{proof}[Proof of \cref{BP_thm:expectation}]
	By a classical argument using Girsanov's theorem\footnote{It also follows from Campbell's formula \cite[sect. 3.2]{Kingman} and \cite[ch. II.1, eq. 2.0.1]{HandbookBrownian}}, $X^\theta_l$ is distributed as the hitting time of level $l$ by a Brownian motion with positive drift $\theta$, hence its density is $
	\frac{d}{dt} \Prob(X^\theta_l \in dt) = y^\theta_l(t) = \idf_{t\geq 0}\frac{e^{-\theta t}\,l\,e^{-l^2/(2t)} }{e^{-\sqrt{2\theta}l}\sqrt{2\pi t^3}}
	$(see \cite[ch. II.1, eq. 2.0.2]{HandbookBrownian}).
	
	Then, going back to the notations of \cref{BP_lem:expectation3}, the joint density of $(\cali P_1,\cali P_2,\cali P_3,\cali P_4)$ at $(t_1,t_2,t_3,t_4) \in \bb (\bb R_+)^4$ equals
	
	\begin{align*}
	&\int_0^\infty d\lambda \sqrt{2\theta} e^{-\sqrt{2\theta}\lambda } y^\theta_{p\lambda/2}(dt_1)y^\theta_{(1-p)\lambda/2}(dt_2)y^\theta_{(1-p)\lambda/2}(dt_3)y^\theta_{p\lambda/2}(dt_4)\\
	&= \frac {\sqrt{2\theta}p^2(1-p)^2}{2^4(\sqrt{2\pi})^4} \frac{e^{-\theta(t_1+t_2+t_3+t_4)}}{(t_1t_2t_3t_4)^{3/2}}
	\int_0^\infty  \lambda^4 e^{-\lambda^2/2\left(\frac{p^2}{4t_1}+\frac{(1-p)^2}{4t_2}+\frac{p^2}{4t_3}+\frac{(1-p)^2}{4t_4}\right)}d\lambda\\
	&= \frac {\sqrt{2\theta}p^2(1-p)^2}{2^4(\sqrt{2\pi})^4} \frac{e^{-\theta(t_1+t_2+t_3+t_4)}}{(t_1t_2t_3t_4)^{3/2}}
	\frac {3\sqrt{2\pi}}{2\left(\frac{p^2}{4t_1}+\frac{(1-p)^2}{4t_2}+\frac{p^2}{4t_3}+\frac{(1-p)^2}{4t_4}\right)^{5/2}}.
	\end{align*}
	
	Now we define the random variables
	$S = \cali P_1+ \cali P_2+\cali P_3+\cali P_4$,
	$Q = \cali P_1/S$,
	$U = (\cali P_1 + \cali P_2)/S$ and
	$V = (\cali P_1 + \cali P_4)/S$. 
	According to \cref{BP_lem:expectation3}, $\E \mu^p$ is the distribution of the pair $(U,V)$.
	It follows from the Lebesgue change of variables theorem that the joint density of $(S,Q,U,V)$ at $(s,q,u,v)\in (\bb R_+ \times \bb R^+ \times [0,1] \times [0,1])$ is equal to
	\begin{equation*}
	\frac{s^3\idf_{\max(0,u+v-1)\leq q \leq \min(u,v)} \frac {3\sqrt{2\theta}p^2(1-p)^2}{2^5(\sqrt{2\pi})^3}e^{-\theta s}}{(sq\,s(u-q)\,s(1-u-v+q)\,s(v-q))^{3/2} \left(\frac{p^2}{4sq}+\frac{(1-p)^2}{4s(u-q)}+\frac{p^2}{4s(1-u-v+q)}+\frac{(1-p)^2}{4s(v-q)}\right)^{5/2}},
	\end{equation*}
	which we rewrite as
	\begin{equation*}
	\left(\frac{\sqrt \theta e^{-\theta s}}{\sqrt \pi \sqrt s}\right)\frac{\frac {3p^2(1-p)^2}{2\pi} \idf_{\max(0,u+v-1)\leq q \leq \min(u,v)}}
	{(q(u-q)(1-u-v+q)(v-q))^{3/2}{\left(\frac{p^2}{q}+\frac{(1-p)^2}{(u-q)}+\frac{p^2}{(1-u-v+q)}+\frac{(1-p)^2}{(v-q)}\right)^{5/2}}}.
	\end{equation*}
	Now we get the joint distribution of $(U,V)$ by integrating with respect to $s$ and $q$, which immediately yields \cref{BP_thm:expectation}.
\end{proof}

\section{Shuffling of continuous trees}
\label{BP_sec:f}

The goal of this section is to build, from a signed excursion $(g,s)$,  a shuffled excursion $f_{g,s}$, that verifies the conclusions of \cref{BP_thm:f} after setting $\tilde e = f_{e,S}$. This will not be possible for every choice of deterministic signed excursion, but we will show that it is possible for signed excursions with property \eqref{BP_eq:olt}, which is the case of $(e,S)$ with probability 1.

We start from the following observation: for every CRT excursion $g$, if we define the $a_i,b_i,c_i,h_i$ as before, then by density of the branching points it is easy to see that
\[g(t) = \sup_{i} h_i \idf_{[a_i,c_i]} (t).\]
Hence, given a sequence of signs $s$, which provides us the numbers $a'_i,b'_i,c'_i$, it is natural to define a shuffled version as such:
\[f_{g,s}(t) = \sup_{i} h_i \idf_{[a'_i,c'_i]} (t)\]

The map $(g,s,t)\mapsto f_{g,s}(t)$ is measurable because the $g(a_i)$, $a'_i$ and $c'_i$ are measurable functions of $g$ and $s$.

From now on, we will drop the dependency in $(g,s)$ in the proofs. So we set $f=f_{g,s}$ and $\varphi = \varphi_{g,s}$.
The first step is to show that $f$ is continuous whenever $(g,s)$ verifies \eqref{BP_eq:olt}. We start with two lemmas. Let $\omega(g,\delta)$ stand for the modulus of continuity of $g$ at radius $\delta$.
\begin{lemma}\label{BP_moduluscty}
	For $a'_k\leq u \leq b'_k$,  $h_k\leq f(u) \leq h_k + \omega(g, b'_k-a'_k)$.
	
	For $b'_k\leq u \leq c'_k$,  $h_k\leq f(u) \leq h_k + \omega(g, c'_k-b'_k)$.
\end{lemma}
\begin{proof}
	The two claims are symmetric, thus only the first is proved.
	Recall that $f(u) = \sup_{[a'_i,c'_i] \ni u} h_i$ and suppose $u\in[a'_k,b'_k]$. 
	For $i$ such that $[a'_i,c'_i] \ni u$, either $h_i\leq h_k$, or $h_i>h_k$. 
	In the latter case, $[a'_i,c'_i] \subset [a'_k,b'_k]$. 
	Hence $\abs{a_i-b_k}<\abs{b'_k-a'_k}$, and $h_i-h_k = g(a_i)-g(a_k) \leq  \omega(g,b_k-a_k) =  \omega(g,b'_k-a'_k) $.
	
	This shows that for every $i$ such that $[a'_i,c'_i] \ni u$, $h_i < h_k+ \omega(g,b'_k-a'_k)$ Taking the supremum gives the claim of the lemma.
\end{proof}

\begin{lemma}\label{BP_densityb'i}
	The $b'_i$, for $i\in \bb N$, are dense in $[0,1]$.
\end{lemma}

\begin{proof}	
	The leaves of $g$ are of full Lebesgue measure.
	If $x$ and $y$ are leaves, there is a $i$ such that $a_i<x<b_i<y<c_i$. 
	As a result of \cref{BP_lem:localization}, $b'_i$ must lie between $\varphi(x)$ and $\varphi(y)$.
	Since $\varphi$ is measure-preserving, the images of leaves of $g$ by $\varphi$ are of full measure, and hence dense in $[0,1]$.
	So the $b'_i$ are dense.
\end{proof}

\begin{proposition}\label{BP_fiscontinuous}
	Under \eqref{BP_eq:olt}, the function $f$ is continuous.
\end{proposition}

\begin{proof}
	Let $t$ be in $[0,1]$ and $\delta >0$. 
	By \cref{BP_densityb'i}, we can find $b'_i<t<b'_j$ with $(b'_j-b'_i)\leq \delta$. 
	Let $k$ be the most recent common ancestor of $i$ and $j$, so that $b'_i<b'_k<b'_j$.	
	We shall show that there is a continuous function $\underline{f}$ such that for $u\in[b'_i,b'_j]$,
	\begin{equation}\label{BP_fsoulign}
	\underline{f}(u)\leq f(u) \leq \underline{f}(u) + \omega(g,\delta)
	\end{equation}
	Which is enough, since $\delta$ was arbitrary, to show continuity in $t$. We build $\underline f$ and show \eqref{BP_fsoulign} on $[b'_k,b'_j]$ only. The interval $[b'_i,b'_k]$ can be treated with a symmetric proof.
	
	Set $\underline f : [b'_k,b'_j] \to \bb R_+$, with 
	\begin{equation*}
	\underline{f}= \sup \{h_l\idf_{[a'_l,c'_l]}  \mid l \colon[a'_k ,c'_k] \supset [a'_l,c'_l] \supset [a'_j,c'_j] \}.
	\end{equation*} 
	Clearly, $\underline{f}\leq f$.
	It is also clear that $\underline{f}$ is increasing from $h_k$ to $h_j$, because the indicator functions are nested and $h_l$ increases as $a'_l$ decreases. \Cref{BP_lem:genealogyofbp2mieux} implies that the $a'_l$ are all distinct, while property \eqref{BP_eq:olt} implies that the $h_l$ are dense in $[h_k,h_j]$. This implies continuity of $\underline f$.
	
	Now we shall show \eqref{BP_fsoulign} for $u$ in $[b'_k,b'_j]$.
	
	\noindent \textit{Case 1: for every $l$ s.t. $u\in[a'_l,c'_l]$,  we have $[a'_l,c'_l] \supset [a'_j,c'_j]$.} Then $f(u) = \underline f(u)$.
	
	\noindent \textit{Case 2: there exists $l$ s.t. $x\in[a'_l,c'_l]$ and $[a'_l,c'_l]\nsupseteq[a'_j,c'_j]$. } Then consider the most recent common ancestor $m$ of $l$ and $j$. Necessarily,
	\[b'_k<a'_m<a'_l<u<c'_l<b'_m<a'_j<c'_j<c'_m.\]
	Then \cref{BP_moduluscty} gives $h_m\leq g(u) \leq h_m + \omega(g,\delta)$. It is clear that $h_m = \underline{f}(u)$, proving \eqref{BP_fsoulign}.	
\end{proof}

Now that we have shown that $f$ is continuous, it becomes possible to define the distance $d_f$ on $[0,1]$ and the structured real tree $\cali T_f$. 

\begin{proposition}\label{BP_phiisometry}
	Under \eqref{BP_eq:olt}, we have $g=f\circ\varphi$, and furthermore, $\varphi$ is a $([0,1],d_g)\to([0,1],d_f)$ isometry.
\end{proposition}
\begin{proof}
	Let $t\in[0,1]$. To show $g(t) = f(\varphi(t))$ it is enough to see that 
	\begin{equation}\label{BP_fundamentalrelationphi}
	\{k : t\in [a_k,c_k]\} = \{k:\varphi(t) \in [a'_k,c'_k]\}.
	\end{equation}
	because $e(t)$ and $f(\varphi(t))$ are just the respective suprema of $i\mapsto h_i$ over these two sets.
	If $k$ is such that $t\in [a_k,c_k]$, then by \cref{BP_lem:localization}, $\varphi(t) \in [a'_k,c'_k]$.
	If on the other hand $k$ is such that $t\notin [a_k,c_k]$, by symmetry suppose $t<a_k$. It is then possible to find $i$ such that $t < a_i <a_k <c_k \leq c_i$.
	Then \cref{BP_lem:localization,BP_lem:genealogyofbp2mieux} imply that $\varphi(t)\notin[a'_k,c'_k]$.
	
	Now to show that $\varphi$ is a $(d_g,d_f)$ isometry, we need only show that for $x<y$, \[\min_{[x,y]} g = \min_{[\varphi(x),\varphi(y)]} f.\]
	
	\noindent \textit{Case 1:} $\min_{[x,y]} g = g(x)$. Then for every $i$, $x\in [a_i,c_i]$ implies $y\in [a_i,c_i]$. So $\varphi(x) \in [a'_i,c'_i]$ implies $\varphi(y) \in [a'_i,c'_i]$ and then $[\varphi(x),\varphi(y)] \subset [a'_i,c'_i]$. The definition of $f$ then yields $f(t)\geq f(\varphi(x))$ for every $t\in [\varphi(x),\varphi(y)]$. Hence 
	\[\min_{[\varphi(x),\varphi(y)]}f = f(\varphi(x))	= g(x) = \min_{[x,y]} g.\]
	
	\noindent \textit{Case 2:} $\min_{[x,y]} g = g(y)$. This case is similar by symmetry.
	
	\noindent \textit{Case 3:} $\min_{[x,y]} g = b_i$ for some $b_i \in (x,y)$.
	Then we conclude immediately by applying case 2 on $[x,b_i]$ and case 1 on $[b_i,y]$.
\end{proof}

\begin{proposition}\label{BP_prop:hisbrownian}
	The random continuous function $f_{e,S}$ has the distribution of a Brownian excursion with the same local times at 1 as $e$.
\end{proposition}
\begin{proof}
	The claim on the local times is an immediate consequence of the fact that for every $y\geq 0$,  $\Leb\{t, f_{g,s}(t)\leq y\} = \Leb\{t, f_{g,s}(\varphi_{g,s}(t))\leq y \} = \Leb \{t, g(t)\leq y \}$.
	
	To show that the random continuous functions $e$ and $f = f_{e, S}$ have the same distribution, we shall show that for every $k\geq 1$, if $U_{(1)}<\ldots<U_{(k)}$ are reordered uniform variables in $[0,1]$, independent of $e,S$, then
	\begin{equation} \label{BP_eq:diste=distf}
	(e(U_{(1)}),\ldots,e(U_{(k)})) \stackrel d = (f(U_{(1)}),\ldots,f(U_{(k)})).
	\end{equation}
	Deriving $e\stackrel d = f$ from there is classical, see for instance the end of the proof of the direct implication of \cite[thm. 20]{aldous1993}.
	
	Let us consider $U_{(1)}<\ldots<U_{(k)}$ the order statistics of $k$ uniform random variables in $[0,1]$, independent of $e,S$. Set $V_i= \varphi(U_{(i)})$ for every $1\leq i \leq k$. Then there exists $\alpha \in \perms_k$ such that $W_1 = V_{\alpha(1)} < \ldots < V_{\alpha(k)} = W_k$. Since $\varphi$ preserves the Lebesgue measure, $(W_1,\ldots,W_k)$ has the distribution of the order statistic of $k$ uniform variables.
	
	We consider the \textit{marked trees}, as per the definition of \cite[sect. 2.5]{legall2005}, associated to a CRT excursion and a finite number of points.  For any set $\mathbf t= (t_1 <\ldots <t_k)$ of leaves of $g$, $\theta(g;\mathbf{t})$ is built from the tree $\tau(g;\mathbf{t})$ by adding edge-lengths compatible with the distances in the tree $\mathcal T_g$. Since the root of $\tau(g;\mathbf{t})$ has a positive height, a new root $\emptyset$ is added below it. It is characterized (among plane trees with edge-lengths up to isomorphism) by the following fact: 
	\begin{equation}\label{BP_eq:caracttree}
	d_{\theta(g;\mathbf{t})}(\ell_i,\ell_j) = d_{g}(t_i,t_j),\quad d_{\theta(g;\mathbf{t})}(\emptyset,\ell_i) = g(t_i),
	\end{equation}
	where $d_{\theta(g;\mathbf{t})}$ denotes the graph distance, taking edge-lengths into account, and in any plane tree $\ell_1,\ldots,\ell_k$ is an enumeration of the leaves in the natural ordering.
	
	Let $T = \theta(e;\mathbf{U})$, and let $\widetilde T$ be obtained from $T$ by inverting the order of the children at each branching point corresponding to a $b_i$ where the sign $s_i$ is a $\ominus$.
	By definition there is an isomorphism of rooted trees with edge-lengths $\widetilde T \leftrightarrow T$.
	This isomorphism necessarily permutes the leaves:
	set $\beta\in \perms_k$ such that $\ell_i(\widetilde T)\leftrightarrow \ell_{\beta(i)}(T)$.
	Then by construction $\beta$ is such that $\varphi_{e,S}(U_{\beta(1)}) < \ldots < \varphi_{e,S}(U_{\beta(k)})$. We deduce $\beta = \alpha$, and hence	
	\begin{align*}d_{\widetilde{T}}({\ell}_i,{\ell}_j) &= d_{T}(\ell_{\alpha(i)},\ell_{\alpha(j)}) = d_e(U_{\alpha(i)}, U_{\alpha(j)}) = d_f(\varphi(U_{\alpha(i)}), \varphi(U_{\alpha(j)})) = d_{\theta(f;\mathbf W)}(\ell_i,\ell_j)\\
	d_{\widetilde{T}}(\ell_i,\emptyset) &= d_{T}(\ell_{\alpha(i)},\emptyset) = g(U_{\alpha(i)}) = g(W_i) = d_{\theta(f;\mathbf W)}(\ell_i,\emptyset).
	\end{align*}
	So $\widetilde T = \theta(f,\mathbf{W})$.
	
	Finally we consider the distribution of $\widetilde T$.
	Theorem 2.11 of \cite{legall2005} tells us that the structure of $T$ is that of a uniform planted binary tree with $k$ leaves, and the edge-lengths are exchangeable.
	So an independent shuffling of $T$ is still distributed like $T$, and this is the case of $\widetilde T$.
	We deduce $\theta(e;\mathbf{U}) = T \stackrel d= \widetilde{T} = \theta(f;\mathbf{W})$.
	From there, \eqref{BP_eq:caracttree} implies that we can recover \eqref{BP_eq:diste=distf}.
\end{proof}

Now \cref{BP_thm:f} follows from \cref{BP_fiscontinuous,BP_phiisometry,BP_olt,BP_prop:hisbrownian}, after setting $\widetilde e = f_{e,S}$.

\section*{Acknowledgements}
I warmly thank Gr\'egory Miermont for his dedicated supervision, enlightening discussions and his detailed reading of this paper. Many thanks to Mathilde Bouvel, Valentin F\'eray and S\'ebastien Martineau for enriching discussions and useful comments. 
Thanks to an anonymous referee for very helpful comments.

I am grateful for the hospitality and support of the Forschungsinstitut f\"ur Mathematik at ETH Z\"urich during a stay where part of this research was conducted. 
\bibliographystyle{plain}
\bibliography{a.bib}

\label{BP_lastpage}
\end{document}